\documentclass[11pt]{article}

\usepackage{amsmath,amssymb,amsthm}
\usepackage{amscd,times}

\newtheorem{thm}{Theorem}
\newtheorem{prop}{Proposition}
\newtheorem{lem}{Lemma}
\newtheorem{cor}{Corollary}

\theoremstyle{remark}
\newtheorem{rem}{Remark}

\newtheorem{ex}{Example}
\newtheorem{pr}{Problem}

\newcommand{\C}{\mathbb{C}}

\newcommand{\lra}{\longrightarrow}

\newcommand{\Q}{\mathbb{ Q}}
\newcommand{\A}{\mathbb{ A}}
\newcommand{\N}{\mathbb{ N}}
\newcommand{\R}{\mathbb{ R}}

\newcommand{\Z}{\mathbb{ Z}}

\title{Moduli space of weighted pointed stable curves and toric topology of Grassmann manifolds}

\author{Victor M.~Buchstaber and Svjetlana Terzi\'c}

\begin{document}

\maketitle
\begin{abstract}
We relate the theory of moduli spaces  $\overline{\mathcal{M}}_{0,\mathcal{A}}$ of stable weighted  curves of genus $0$  to the equivariant topology of complex Grassmann manifolds $G_{n,2}$, with  the canonical action  of the compact torus $T^n$.  We prove that all spaces $\overline{\mathcal{M}}_{0,\mathcal{A}}$  can be isomorphically or up to birational morphisms   embedded in $G_{n,2}/T^n$ . 
The crucial role for proving this result  play  the chamber decomposition of the hypersimplex $\Delta _{n,2}$ which corresponds to $(\C ^{\ast})^{n}$-stratification of $G_{n,2}$ and the spaces of parameters over the  chambers, which are subspaces in $G_{n,2}/T^n$. 
We  show that  the points  of   these moduli spaces  $\overline{\mathcal{M}}_{0, \mathcal{A}}$  have the geometric realization as the points of the spaces of parameters over  the  chambers. 
We single out  the characteristic   categories among  such moduli spaces.  The morphisms in  these  categories correspond to the natural projections between the universal space of parameters and  the spaces of parameters over the chambers.
As  a corollary, we obtain the  realization of the orbit space $G_{n,2}/T^n$  as an  universal object for the introduced categories. As one of our main results  we describe the structure of the canonical projection from the Deligne-Mumford compactification to the Losev-Manin compactification of $\mathcal{M}_{0,n}$,  using the  embedding of $\mathcal{M}_{0, n}\subset \bar{L}_{0, n, 2}$ in $(\C P^{1})^{N}$, $N=\binom{n-2}{2}$, the action of the algebraic torus $(\C ^{\ast})^{n-3}$ on $(\C P^{1})^{N}$ for which $\bar{L}_{0, n, 2}$ is invariant,   and the realization of  the Losev-Manin compactification as the corresponding   permutohedral toric variety.
\end{abstract}

\tableofcontents

\section{Introduction}

In his  seminal paper~\cite{K} Kapranov obtained the results on toric geometry of positive complexity by proving that the moduli space $\overline{\mathcal{M}}_{0.n}$ of stable $n$-pointed curve coincides with the Chow quotient $G_{n,2}\!/\!/(\C ^{\ast})^{n}$ for the canonical action of the algebraic torus $(\C ^{\ast})^{n}$ on a complex Grassmann manifold $G_{n,2}$ of the  two-dimensional complex subspaces in $\C ^{n}$. In our paper~\cite{BT2} we established a connection between the space $\overline{\mathcal{M}}_{0, n}$ and the results of   toric  topology of positive complexity  by proving that  $\overline{\mathcal{M}}_{0, n}$ coincides with the universal space of parameters $\mathcal{F}_{n}$ for the canonical  compact torus   $T^n$  - action on $G_{n,2}$.  In this paper we extend this result and establish a connection  between   the moduli spaces of $\mathcal{A}$ - weighted stable curves  $\overline{\mathcal{M}}_{0, \mathcal{A}}$ of genus $0$  and equivariant  topology of $T^n$-action on $G_{n,2}$.

The canonical action of the compact torus $T^n$ on  $G_{n,2}$  gives wonderful class of examples of a torus action of positive complexity and the image of the corresponding  standard moment map is the hypersimplex $\Delta _{n,2}$. This action can be extended to the canonical action of an algebraic torus $(\C ^{\ast})^{n}$ which induces the standard  well behaved stratification of $G_{n,2}$ for which the  strata consist of those $(\C ^{\ast})^{n}$-orbit which have the same image polytope by the moment map~\cite{GS},~\cite{BT1},~\cite{BT2}. These polytopes produce the chamber decomposition  of $\Delta _{n,2}$  which has important role in the description of the orbit space $G_{n,2}/T^n$. It is firstly observed in~\cite{GMP} that the preimage of a point from a chamber   by the induced moment map $\hat{\mu} : G_{n,2}/T^n \to \Delta _{n,2}$  does not depend on the choice of a point. This leads to the notion of the space of parameters of a chamber, it is a subspace in $G_{n,2}/T^n$ and  it is   a smooth manifold for a chamber of  the maximal dimension~\cite{BT}.

The purpose of this paper  is to relate the  theory of moduli spaces of stable weighted curves  of genus $0$ to $T^n$ - equivariant topology of the Grassmannians $G_{n,2}$. Precisely, we   show  that the spaces of parameters of the chambers  in $\Delta _{n,2}$ for the canonical $T^n$-action on $G_{n,2}$ can be realized, in general case up to a birational morphism,  by the moduli spaces of stable pointed weighted curves of genus $0$. Moreover, the spaces of parameters which correspond to the chambers of maximal dimension are isomorphic to such appropriate moduli spaces. We prove the vise versa as well,  the moduli spaces of weighted pointed genus zero curves can be realized by the spaces of parameters of    the chambers, up to   birational morphisms.

In this way we prove  that the orbit spaces $G_{n,2}/T^n$ provide a kind of universal  geometrical  object  for the theory of moduli spaces of stable weighted curves of genus $0$.  We elaborate  this in more detail.

\subsection{Moduli spaces} Problem of  moduli compactification of  the moduli spaces $\mathcal{M}_{g, S}$  of irreducible smooth $S$ - pointed curves    $\{(C, s_{i}| i\in S)\}$ of genus $g$, where $\{s_i\}$  is a family of closed points on $C$, which are not necessarily pairwise distinct, is widely known and of extreme importance in many areas of mathematics and mathematical physics.  The fundamental result in this direction  is that the moduli space of  the families of stable $S$-pointed curves of genus $g$ has the structure of the   proper smooth connected Deligne-Mumford stack $\overline{\mathcal{M}}_{g,S}$,  which contains $\mathcal{M}_{g,S}$ as an open dense substack and it  is its compactification,~\cite{DM}. 

The moduli space of stable pointed weighted curves  $\overline{\mathcal{M}}_{g, \mathcal{A} (S)}$  was introduced and studied in~\cite{H}. Hassett assigned to marked  points $\{s_i | i \in S\}$ additional parameters $a_{i}(s_i)\in \Q, 0 <  a_{i}(s_i) \leq 1$, that is weights,  which led to a  new family of stability conditions. He proved that the moduli spaces of  such weighted stable $S$-pointed curves give   new Deligne-Mumford stacks and that they  provide new  different compactifications of $\mathcal{M}_{g,S}$.  When $a_{i}(s_i) = 1$ for each $i\in S$ the moduli space   $\overline{\mathcal{M}}_{g, \mathcal{A} (S)}$ coincides with the classical Deligne-Mumford moduli space $\overline{\mathcal{M}}_{g,S}$.

Prior Hassett's paper the moduli spaces $\overline{L}_{g, S}$  and their moduli stack structure   were introduced by Losev  and Manin in~\cite{LM}. The stack $\overline{L}_{g,S}$  parametrizes algebraic
curves of genus $g$ endowed with a family of smooth, painted by black or white,  points,  labeled by S satisfying
a certain stability conditions, that is all white points are pairwise distinct and distinct from black ones, while black points may behave  arbitrarily.  Manin showed  in~\cite{M} that the stacks $\overline{L}_{g, S}$ are special case of Hassett's stacks  corresponding to the weight data given by $a_{i}(s_i) = 1$ for a white point $s_i$, while the sum of   weights corresponding to the black points is $\leq 1$.  The   spaces $\overline{L}_{0, n,2}$ determined by  genus zero curves with exactly two white points can  be as well exhibited as  the known toric varieties over permutohedron. Any of these  toric varieties is, by the result of Gelfand-Serganova~\cite{GS},  known to be   isomorphic to the closure
of a generic algebraic torus orbit  in the space of complete complex flags.  As a generalization of this result,  we mention  the result of~\cite{FJR},  which states   that the toric  variety obtained as the blow-up of a projective space along coordinate subspaces corresponding to the connected subgraphs of a graph $\Gamma$  is isomorphic to a Hassett space $\overline{\mathcal{M}}_{0, \mathcal{A} (S)}$ exactly in the case when $\Gamma$ is an  iterated cone over a discrete set.

After being introduced, the Hassett spaces have been intensively studied.  We mention the result   proved  in~\cite{MM}  that the automorphism groups of the spaces 
$\overline{\mathcal{M}}_{g,S}$ for $g\geq 1$ are isomorphic to  subgroups of the symmetric group $S_n$, $n=|S|$, while for $g=0$ it is described the automorphism group of Hassett spaces which appear  in  Kapranov's blow-up construction of $\overline{\mathcal{M}}_{0,n}$.  Then in~\cite{BC} it is considered and resolved the combinatorial relation between  the intersection theory of $\psi$-classes of the Hassett spaces   $\overline{\mathcal{M}}_{g, \mathcal{A} (S)}$  and the spaces $\overline{\mathcal{M}}_{g,S}$.  Hassett spaces were further studied from the point of view of their  relations to geometric invariant theory quotients. Namely, Hassett proved in~\cite{H} that for $g=0$  and for $\mathcal{A}(S)$ being in an appropriate neighborhood of a typical linearisation $\mathcal{L}$, the spaces   $\overline{\mathcal{M}}_{0,\mathcal{A}(S)}$  are isomorphic to GIT-quotient of  $(\C P ^{1})^{n}$ by the diagonal action of $PGL_{2}(\C)$, with respect to the linearisation    $\mathcal{L}$.  Further, it is proved  in~\cite{KM} that  for $a_{i}(s_i) = \varepsilon$,  $i\in S$ the space  $\overline{\mathcal{M}}_{0,\mathcal{A}(S)}$ can be obtained by a sequence of explicit blow-ups from the GIT quotient  $(\C P ^{1})^{n}$ by the diagonal action of $PGL_{2}(\C)$, with respect to the symmetric  linearisation    $\mathcal{L} = (1, \ldots ,1)$. In addition, complementing the work of Keel~\cite{K} on homology structure of $\overline{\mathcal{M}}_{0, n}$ , the homology groups, and thus  the Chow groups,  of the spaces  $\overline{\mathcal{M}}_{0, \mathcal{A} (S)}$ are described   by Ceyhan in~\cite{C}, although by the  different method.

Hassett theory of weighted pointed stable curves has been developed in many directions. We mention the paper by Bayer and Manin~\cite{BM} in which this theory is extended to the weighted stable maps, then the paper by Ulrisch~\cite{U} in which a tropical analogue  of Hassett spaces was  defined and studied, and the paper by Alexeev~\cite{A}  which gives common generalization of Hassett spaces and stable hyperplane arrangement defined in~\cite{HKT} to weighted stable hyperplane arrangements.

\subsection{$T^n$-equivariant topology of $G_{n,2}$} In our series of papers~\cite{BT0},~\cite{BT1},~\cite{BT2} we have studied the orbit spaces $G_{n,2}/T^n$ for the canonical action of the compact torus $T^n$ on the complex Grassmann manifolds $G_{n,2}$ of two-dimensional complex  subspaces in $\C ^{n}$. The complexity of this action is $d= 2(n-2) - (n-1) = n-3 $. For $n=4$, we proved in~\cite{BT} that $G_{4, 2}/T^4 \cong S^5$, while for $n=5$ and $n=6$, due to homology structure, it follows  from ~\cite{BT1} and ~\cite{IT}, respectively,  that  $G_{n,2}/T^n$ is not a manifold.  This is in the line with the fact that the study of  torus actions of  complexity $d\geq 2$ is considered in the literature  as a quite difficult problem.  

We approached this problem in~\cite{BT1} by constructing  the model $(U_n, G_n)$  for the orbit space $G_{n,2}/T^n$, where $U_{n}= \Delta _{n,2}\times \mathcal{F}_{n}$ for a hypersimplex $\Delta _{n,2}$ and a smooth compact manifold $\mathcal{F}_{n}$, and $G_{n} : U_n\to G_{n,2}/T^n$ is a  continuous projections . This projection h is a diffeomorphism between open, smooth, dense  submanifolds in $U_n$ and $G_{n,2}/T^n$, while for the manifold $\mathcal{F}_{n}$ we prove to be diffeomorphic to the Deligne-Mumford compactification $\overline{\mathcal{M}}_{0, n}$.  This model  incorporates several structural data arising from the  stratification of $G_{n, 2}$,  which corresponds to the canonical action of  the algebraic torus $(\C ^{\ast})^{n}$ on $G_{n,2}$.  
In particular,  in this paper an important role have the spaces of parameters of the chambers arising from the matroidal  decomposition of the hypersimplex $\Delta _{n,2}$ determined by the $(\C ^{\ast})^{n}$ on $G_{n,2}$. 

\subsection{The results}

 In this paper we prove that the  structural data  of our  model $(U_n, G_n)$ can be realized  in terms of the Hassett theory, that is in terms of the spaces  $\overline{\mathcal{M}}_{0, \mathcal{A} (S)}$ and morphisms between them.  Furthermore,   within the Hassett spaces, we  introduce the category, which we call the Hassett category, and prove that it can be isomorphically  modeled by the structural data of our model $(U_n, G_n)$.  In addition, we consider the category of Losev-Manin spaces for which we also prove that  can be modeled on $(U_n, G_n)$, up to birational morphisms.

As the  final result we relate all previously  metioned     to the orbit spaces $G_{n,2}/T^n$ and show that the objects from the Hassett category can be realized by  smooth manifolds in $G_{n,2}/T^n$. For the objects from the  Losev-Manin category, as well as for  any Hassett space, we prove that can  be realized as well   by smooth manifolds in $G_{n,2}/T^n$, but up to birational morphisms.  This proves that the orbit spaces $G_{n,2}/T^n$, although  having itself quite complicated, even topological structure, comprise a lot of important manifolds, among which are, up to a birational morphism,  some special graph associahedra toric varieties.

\subsubsection{ Deligne-Mumford and Losev-Manin compactification of $\mathcal{M}_{0,n}$} As an  application of the mentioned general results  we devote   special attention to the Losev-Manin spaces $\bar{L}_{0, n,2}$ due to  their importance in mathematical physics, but also due to the  fact that they can be exhibited as the permutohedral toric varieties, so  they   provide a  class of toric  moduli spaces which are  compactification of the moduli spaces $\mathcal{M}_{0,n}$.   We prove that the spaces $\bar{L}_{0, n, k}$ can be embedded in $G_{n,2}/T^n$ up to birational morphisms. In the case $k=2$ we describe an embedding  of $\mathcal{M}_{0, n}\subset \bar{L}_{0, n,2}$   in $(\C P^{1})^{N}$, $N = \binom{n-2}{2}$, define the action of the algebraic torus $(\C ^{\ast})^{n-3}$ on $(\C P^{1})^{N}$ for which $\bar{L}_{0, n 2}$ is invariant, describe an extension of the image of $\mathcal{M}_{0, n}\subset \bar{L}_{0, n,2}$ to the principal orbit of the defined $(\C ^{\ast})^{n-3}$- action, and after that we consider the  toric compactification of that orbit.  In this way we obtain explicit realization of a toric variety which coincides with the Losev-Manin compactification $\bar{L}_{0, n,2}$.  A such toric variety is  the closure  in $(\C P^{1})^{N}$ of the space  of parameters $F_n$ of the main stratum in $G_{n,2}$ and it serves   as the starting variety for the wonderful compactification leading to the Deligne -Mumford compactification $\overline{\mathcal{M}}_{0, n}$. Moreover, we show that the outgrows of the Losev-Manin compactification of $\mathcal{M}_{0, n}$ can be divided into two classes, such that the outgrows from the first class provide extension of $\mathcal{M}_{0, n}$ to the  principal $(\C ^{\ast})^{n-3}$-orbit, while the outgrows from the second class provide toric compactification of this principal orbit to $\bar{L}_{0, n, 2}$. The  outgrows from the two classes we explicitly characterize.  This leads to the  description of the  structure of the projection, that is the reduction morphism from the Deligne-Mumford compactification $\overline{\mathcal{M}}_{0, n}$  to the Losev-Manin  compactification $\bar{L}_{0, n,2}$.  In addition,  we relate 
the combinatorics of the toric outgrows in the Losev-Manin compactification  to the inversion formula problem for the power series, following the role the Deligne-Mumford outgrows  play  in this problem.

\section{Moduli spaces of  weighted pointed stable curves}

The seminal paper~\cite{H} is concerned with  the moduli space of stable pointed weighted curves. A moduli space $\mathcal{M}_{g}$  of a complete non singular curves of genus $g$ admits compactification $\mathcal{M}_{g}\subset \overline{\mathcal{M}}_{g}$ by stable curves.  The same holds for a moduli space $\mathcal{M}_{g,n}$ of a complete non singular curves with $n$ marked distinct points, it admits compactification $\overline{\mathcal{M}}_{g,n}$ by stable curves  with $n$ marked points. These are the curves $(C, S=\{s_1, \ldots ,s_n\})$ of genus $g$ such that their only singularities are nodes, marked points are distinct and non-singular and there are finitely many automorphisms of $(C, S)$. The third condition is equivalent to the requirement that $K_{C}+\sum _{i=1}^{n}s_i$ is an ample divisor, where $C$ is the canonical class of $C$. These moduli problems of moduli compactifications of moduli spaces,  are generalized in~\cite{H}  by  considering pointed curves as pairs $(X, D)$, where $X$ is a variety and $D=\sum _{i}^{n}a_iD_i$ is an effective $\Q$-divisor on $X$.  For $\dim X=1$  and $D$ being reduced, the resulting moduli compactification  is the Grotendieck-Knudsen-Deligne-Mumford moduli space of pointed stable curves.

We follow~\cite{H} for the formal definitions.   

For a  genus $g$ curve a  weight data $\mathcal{A}$  is given by an element $(a_1,\ldots , a_n)\in \Q ^{n}$ such that $0<a_i\leq 1$, $i=1, \ldots , n$ and  $2g-2+a_1+\ldots +a_n>0$.  

A connected $S$-pointed curve $(C, S)$ of genus $g$  is said to be
weighted stable  if the following conditions are satisfied: $K_{C} + \sum _{i=1}^{n}a_is_i$ is an ample divisor and for any  subset $I \subset  S$  such that $s_i$ pairwise coincide for $i\in I$ it holds $\sum _{i\in I}a_i \leq 1.$ This definition has a direct extension families.  
Let $n $ and $g$ be non-negative integers and $B$ a noetherian scheme. A family of nodal curves of genus $g$ with $n$ marked points is given by
\begin{itemize}
\item a flat proper morphism $\pi : C \to B$ whose geometric fibers are nodal connected curves of genus $g$;
\item sections $s_1,\ldots , s_n$ of $\pi$.
\end{itemize}


A family of nodal curves of genus $g$ with $n$ marked points is stable of type  $(g, \mathcal{A})$ if
\begin{itemize}
\item the sections $\{s_1, \ldots, s_n\}$  lie in the smooth locus of $\pi $ and for any subset $\{s_{i_1}, \ldots, s_{i_r}\}$ with non-empty intersections it holds $a_{i_1}+\ldots +a_{i_r}\leq 1$;
\item $K_{\pi} + a_1s_1+\ldots +a_ns_n$ is $\pi$-relatively ample divisor.
\end{itemize} 

For $a_1=\ldots =a_n=1$ this coincides with the traditional notion of pointed stable curves.  In~\cite{H} it is proved the following theorem which solves the mentioned moduli problem.

\begin{thm}
Let $(g, \mathcal{A})$ be a collection of input data, where $n\geq 3$.  There exists a connected Deligne-Mumford stack $\overline{\mathcal{M}}_{g, \mathcal{A}}$, smooth and proper over $\Z$, representing the moduli problem of pointed stable curves of type $(g, \mathcal{A})$.
\end{thm}

The construction of these moduli stacks carries some natural morphisms between them. More precisely, in~\cite{H} the following statements on existence  of reduction and forgetting morphisms are proved respectively:

\begin{thm}\label{reduction}
Fix $g$ and $n$ and let $\mathcal{A} = (a_1, \ldots , a_n)$ and $\mathcal{B} = (b_1, \ldots , b_n)$ be collections of weight data such that $b_i\leq a_i$ for $i=1,\ldots n$. Then there exists birational reduction morphism 
\[
\rho _{\mathcal{A}, \mathcal{B}} : \overline{\mathcal{M}}_{g, \mathcal{A}} \to   \overline{\mathcal{M}}_{g, \mathcal{B}}.
\]
For an element $(C, s_1,\ldots s_n)\in  \overline{\mathcal{M}}_{g, \mathcal{A}}$, $\rho _{\mathcal{A}, \mathcal{B}}(C, s_1,\ldots ,s_n)$ is obtained by successively collapsing components of $C$ along which $K_{C}+b_1s_1+\ldots +b_ns_n$ fails to be ample.
\end{thm}

\begin{thm}
Fix $g$ and $n$ and let $\mathcal{A}$ be a collection of weight data and $\mathcal{A}^{'} = \{a_{i_1}, \ldots , a_{i_r}\}$ a subset such that $2g-2+a_{i_1}+\ldots +a_{i_r}>0$.  Then there exists natural forgetting   morphism 
\[
\Phi _{\mathcal{A}, \mathcal{A}^{'}} : \overline{\mathcal{M}}_{g, \mathcal{A}} \to   \overline{\mathcal{M}}_{g, \mathcal{A}^{'}}.
\]
For an element $(C, s_1,\ldots s_n)\in  \overline{\mathcal{M}}_{g, \mathcal{A}}$, $\Phi _{\mathcal{A}, \mathcal{A}^{'}}(C, s_1,\ldots ,s_n)$ is obtained by successively collapsing components of $C$ along which $K_{C}+a_{i_1}s_{i_1}+\ldots +a_{i_r}s_{i_r}$ fails to be ample.
\end{thm}

In addition, it is described in~\cite{H} the exceptional locus of the reduction morphism by the following statement.

\begin{prop}\label{reductionboundarydivisors}
The reduction morphism $\rho _{\mathcal{A}, \mathcal{B}} : \overline{\mathcal{M}}_{g, \mathcal{A}} \to   \overline{\mathcal{M}}_{g, \mathcal{B}}$ contracts the boundary divisors 
\begin{equation}\label{decomprule}
D_{I, J} = \overline{\mathcal{M}}_{0, \mathcal{A}_{I}^{'}} \times \overline{\mathcal{M}}_{g, \mathcal{A}_{J}^{'}},
\end{equation}
where
\[ \mathcal{A}_{I}^{'} = (a_{i_1},\ldots , a_{i_r}, 1), \;\; \mathcal{A}_{J}^{'} =(a_{j_1},\ldots , a_{j_{n-r}},1),
\]
which correspond to the partitions 
\[
I\cup J = \{1,\ldots ,n\}, \;\; I=\{i_1,\ldots , i_r\}, \;\; J=\{j_1,\ldots , j_{n-r}\}
\]
determined by $b_{I} = b_{i_1}+\ldots + b_{i_r}\leq 1$ and $2<r\leq n$. There is a  factorization of $\rho _{\mathcal{A}, \mathcal{B}} | D_{I,J}$ as
\[
\overline{\mathcal{M}}_{0, \mathcal{A}_{I}^{'}}\times \overline{\mathcal{M}}_{\mathcal{A}_{J}^{'}} \stackrel{\pi}{\to}\overline{\mathcal{M}}_{g, \mathcal{A}_{J}^{'}}\stackrel{\rho}{\to} \overline{\mathcal{M}}_{g, \mathcal{B}_{J}^{'}},
\]
where $\mathcal{B}^{'}_{J}= (b_{j_1},\ldots , b_{j_{n-r}}, b_{I})$,  $\rho = \rho _{\mathcal{A}^{'}_{I}, \mathcal{B}^{'}_{J}}$ and $\pi$ is the projection.
\end{prop}
 
\begin{rem}
Note that Propositon~\ref{reductionboundarydivisors} recursively describes  the boundary divisors by
\[
 \overline{\mathcal{M}}_{0, \mathcal{A}_{I}^{'}} \times  \overline{\mathcal{M}}_{0, \mathcal{A}_{J}^{'}} \to \overline{\mathcal{M}}_{0, \mathcal{A}_{I}^{'}}\times \overline{\mathcal{M}}_{0, \mathcal{A}_{I}^{"}}\times \overline{\mathcal{M}}_{0, \mathcal{A}_{J}^{"}}
\to \cdots,
\]
where $\overline{\mathcal{M}}_{0, \mathcal{A}_{I}^{"}}\times \overline{\mathcal{M}}_{0, \mathcal{A}_{J}^{"}}$ comes from $ \overline{\mathcal{M}}_{0, \mathcal{A}_{J}^{'}}$ following the decomposition rule~\eqref{decomprule}.
\end{rem}

It follows from the definition of weighted pointed stable curves of genus $g$ that  the domain $\mathcal{D}_{g,n}$ of admissible weight data is
\[
\mathcal{D}_{g,n} = \{(a_1,\ldots , a_n)\in \R^{n} | 0<a_j\leq 1 \;  \text{and}\; a_1+\ldots + a_n > 2-2g\}.
\]
Let $\mathcal{W}$ be a hyperplane arrangement in  $\R ^{n}$,  which consists of a finite number of hyperplanes such that for each $w\in \mathcal{W}$ it holds   $w\cap \mathcal{D}_{g,n}\neq\emptyset$.  The chamber decomposition of $\R ^{n}$ defined by $\mathcal{W}$ induces the chamber decomposition of $\mathcal{D}_{g,n}$. 

The coarse chamber decomposition of $\mathcal{D}_{g,n}$  is  defined by the hyperplane arrangement
\begin{equation}\label{chcoarse}
\mathcal{W}_{c} = \big\{\sum\limits_{j\in S} a_j=1 | S\subset \{1,\ldots ,n\}, 2<  |S|< n-2\big\},
\end{equation}
while its fine chamber decomposition is defined  by the hyperplane arrangement
 
\begin{equation}\label{chcoarse}
\mathcal{W}_{f} = \big\{\sum\limits_{j\in S} a_j=1 | S\subset \{1,\ldots ,n\}, 2\leq  |S|\leq n-2\big\}.
\end{equation}
Obviously, the fine chambers are contained in the coarse chambers, so whatever holds for the second ones holds for the first  ones as well. 

The coarse chamber decomposition has the following important property~\cite{H}:
\begin{prop}\label{chambersweight}
The coarse chamber decomposition is the coarsest decomposition of $\mathcal{D}_{g,n}$ such that $\overline{\mathcal{M}}_{g,\mathcal{A}}$ is constant on each chamber.
\end{prop}

\section{Relation to  GIT quotients of $(\C P^{1})^{n}$}

\subsection{General facts about Geometric invariant theory}
We recall,   following mainly~\cite{dolg} and~\cite{Th},  some well known notions and results from  the geometric  invariant theory in order to make exposition more self contained.

\subsubsection{$G$-linearised line bundles}
Let $X$ be a variety with an  action of an algebraic group $G$. A line bundle $L$ on $X$     is said to be  $G$ - linearised if it  admits    $G$ - action $\sigma $  such that $\pi  \circ \sigma (g) = g\cdot \pi$ , where $\pi : L\to X$  is the projection.   Such a bundle on $G$-variety $X$  always exists, in fact, for any line bundle $L$ on $X$  the tensor power $L^{\otimes k}$ for some $k\in \N$ will be a  $G$-linearised line bundle. In addition if  $X$ is a  quasi - projective variety then it admits an  ample $G$-linearised line bundle.

\subsubsection{Picard and N\'eron-Severi groups, fractional linearisations}
Recall that in the Picard group  for $X$, denoted by $\text{Pic}(X)$, of isomorphism  classes of line bundles, the ampleness property of a bundle  depends only on the algebraic equivalence class of a bundle. The group of algebraic equivalence classes of line bundles in $\text{Pic}(X)$  is known as N\'eron-Severi group $\text{NS}(X)$, it is finitely  generated and contains an ample subset $A(X)$. It can be defined the ample cone by $A_{\Q}(X)=A(X)\otimes _{\N}\Q \subset \text{NS}_{\Q}(X)= \text{NS}(X)\otimes \Q$, where $\text{NS}_{\Q}(X)$ is a finite dimensional rational vector space.  The group of isomorphism classes of  $G$ - linearised line bundles on $X$ is denoted by  $\text{Pic}^{G}(X)$ and the kernel of the forgetful morphism $\text{Pic}^{G}(X)\to \text{Pic}(X)$ is given by the group of characters $\chi (G)$. There is a $G$-algebraic equivalence relation on $\text{Pic}^{G}(X)$ analogous to an algebraic equivalence on $\text{Pic}(X)$ and it produces the group $\text{NS}^{G}(X)$.  This group  is finitely generated as well  and $\text{NS}^{G}_{\Q}(X) = \text{NS}^{G}(X)\otimes \Q$ is a finite dimensional vector space. An element of $\text{NS}_{\Q}^{G}(X)$ is called a fractional linearisation. The map $f : NS ^{G}(X)\to NS(X)$ is not surjective in general, but $f_{\Q} : \text{NS}^{G}_{\Q}(X) \to \text{NS}_{\Q}(X)$ is an epimorphism whose kernel  
is $\chi (G)\otimes \Q$, the group of fractional characters.

For example,  given  $G$-linearised ample line bundles $L_1, \ldots , L_n$ on a  $G$-variety $X$, then for $(t_1, \ldots , t_n)\in \Q^{n}$ the tensor product $L(t_1, \ldots, t_n) = \otimes _{i}L_{i}^{t_i}$ will be  a fractional linearisation on $X$.  

An ample linearisation $L$  on $X$  is said to be $G$-effective if  $L^{n}$ has a $G$-invariant section for some $n>0$. This notion depends only on the $G$-algebraic equivalence class of the linearisation, which implies that it is well defined $G$-effective subset $E^{G}(X)\subset f^{-1}(A(X)) \subset  NS ^{G}(X)$ and $G$-effective cone $E_{\Q}^{G} = E_{G}\otimes _{\N}\Q_{\geq 0}\subset NS_{\Q}^{G}$.

\subsubsection{Categorical and GIT quotients}
It is the result of Hilbert following by Nagata that for a given vector space $V$ and a reductive linear algebraic group $G$  acting on it, the commutative algebra $A$  of invariant polynomials on $V$ is finitely generated and any set of generators $P_1, \ldots , P_{N}$  for $A$ defines invariant regular map  $f$ from $V$ to some affine algebraic variety $Y$ belonging to an  affine space $\A ^{N}$, whose ring of polynomials is isomorphic to $A$.   This map has universal property for all $G$-invariant maps of $V$ and it is called  the categorical quotient of $V$ for the action $G$. Passing to the projective space $P(V)$, one considers  the  induced image of the set of semi-stable points  $V^{ss}$, that is those points in $V$ which do not belong to the preimage by $f$ of the origin. It is  an invariant open subset $\mathbb{P}(V)^{ss}$ in $\mathbb{P}(V)$ and the above map induces the map from $\mathbb{P}(V)^{ss}$ to the projective algebraic variety denoted by $\mathbb{P}(V)^{ss}\!/\!/G$ whose projective coordinate algebra is isomorphic to $A$.  The fibers of the  map $\mathbb{P}(V)\to \mathbb{P}(V)^{ss}\!/\!/G$ generally  are not  orbits, but each fiber contains an unique closed orbit, which means  $\mathbb{P}(V)^{ss}\!/\!/G$ parametrizes closed orbits in $\mathbb{P}(V)^{ss}$. The stable points $\mathbb{P}(V)^{s}$  in $\mathbb{P}(V)^{ss}$ are those points whose orbits are closed and whose stabilizer subgroups are finite.   The restriction of the map $\mathbb{P}(V)^{ss}\to \mathbb{P}(V)^{ss}\!/\!/G$ to the stable points is an orbit map $\mathbb{P}(V)^{s} \to  \mathbb{P}(V)^{s}/G$   and it is called a  geometric quotient. 

The strategy for defining GIT quotient of a projective variety $X$  with a $G$-action is to cover it by open affine $G$-invariant subsets $V_i$ and to construct categorical quotient by gluing together the quotients $V_i\!/\!/G$. This can not be done on the whole  $X$, so this distinguish an open invariant subset $U$ of $X$ for which it can be done. This leads to consider semi-stable and stable points in $X$ related to a  $G$-linearised line bundle $L$.

A point $x\in X$ is called semi-stable  if there exists $n>0$ and $s\in \Gamma (X, L^{n})^{G}$ such that $X_{s}=\{y\in X : s(y)\neq 0\}$ is affine and contains $x$. A point $x\in X$ is called stable if it is semi-stable and additionally $G_x$ is finite and all orbits of $G$ in $X_s$ are closed. The set of semi-stable points is denoted by $X^{ss}(L)$ and of stable points by $X^{s}(L)$. If $X$ is projective and $L$ is ample any point is semi-stable.

Let $L$ be a $G$-linearised ample line bundle on a variety $X$.  This action restricted to   $X^{ss}\subset X$ can  be linearised, that is there exists an $G$-equivariant embedding $X^{ss}\to \C P^n$, where $G$ acts on $\C P^n$ by linear representation $G\to GL_{n+1}(\C)$.  More precisely, the group $G$ acts naturally and linearly on the vector space of sections $\Gamma (X, L)$ by $(g\cdot s)(x) = g\cdot s(g^{-1}\cdot x)$ and   there exists a finite set of invariant sections $\{s_1,\ldots, s_n\}\in \Gamma (X, L^{\otimes N})$ for a  sufficiently large $N$.   If  $V\subset \Gamma (X, L^{\otimes N})$ is a vector subspace spanned by $\{s_0, \ldots , s_n\}$,  then it is obtained a $G$-equivariant embedding $X^{ss}\to P(V)$ defined by $x\to (s_{1}(x), \ldots , s_{n}(x))$. 

In addition, sections $\{s_0, \ldots, s_n\}$ can be chosen such that  $X^{ss}$ is covered by $X_{s_i}$.  Since $X_{s_i}$ are affine, there are defined    categorical quotients $X_{s_i}\!/\!/G$, and ampleness of $L$ enables to prove that these categorical quotients  are glued together to give a variety. This  variety  is called GIT quotient of $X$ by $G$ and denoted by $X\!/\!/G$.

The set $X^{ss}(L)$ as well as the GIT quotient  $X\!/\!/G$ depend only on the $G$-algebraic equivalence class of $L$. Moreover, the dependence of  the quotient  $X\!/\!/G$ on the choice of $L\in NS^{G}$ is as follows: the cone $E_{\Q}^{G}(X)$ is  locally polyhedral in  $f_{\Q}^{-1}(A_{\Q}(X))$, it is divided by the finite number of  locally polyhedral walls of codimension one into convex chambers , such that as $t$ varies within a fixed chamber, the set $X^{ss}(t)$ and the quotient $X\!/\!/G(t)$ remain fixed, while for $t_0$ on a wall or walls or on the boundary of  $E_{\Q}^{G}(X)$, and $t_1$ on an    adjacent chamber there is an inclusion $X^{ss}(t_1)\subset X^{ss}(t_0)$, which induces a canonical projective morphism    $X\!/\!/G(t_1)\to   X\!/\!/G(t_0)$.

\subsection{ Diagonal $PGL_{2}(\C)$-action on  $(\C P^{1})^{n}$}
Let  $X=(\C P^{1})^{n}$ is  considered  with the diagonal action of $PGL_{2}(\C)$.  Let $L_i= \mathcal{O}(-1)$  be  the canonical line bundle on  the $i$-the factor $\C P^{1}$.  For a $(t_1, \ldots ,t_n)\in \Q ^{n}$,   the fractional linearisation $\mathcal{O}(t_1, \ldots , t_n) = \otimes _{i=1}^{n}L_{i}^{t_i}$ is an ample line bundle on $(\C P^{1})^{n}$.  

For this action, the  semi-stability and stability notion of a point related to a fractional linearisation $\mathcal{O}(t_1, \ldots , t_n)$  can be, following~\cite{Th}, reformulated as follows:

A point $(x_1,\ldots , x_n)\in (\C P^{1})^{n}$ is said to be stable (semi-stable) if for each $x\in \C P^{1}$  it holds
\[
\sum_{j=1}^{n}t_j\delta (x, x_j) < (\leq ) \frac{1}{2}\sum_{j=1}^{n}t_j,
\]
where $\delta (x, x_j)=1$ for $x=x_j$ and $0$ otherwise.

The stability condition says: a  point $(x_1, \ldots , x_n)$ is stable if  for any $\{i_1, \ldots i_r\}\subset \{1,\ldots , n\}$, $x_{i_1},\ldots , x_{i_r}$ may coincide only when $x_{i_1}+\ldots +x_{i_r} <1$.

Let $\mathcal{O}(\mathcal{L})$ for  $\mathcal{L}  = (t_1, \ldots , t_n)\in \Q ^{n}$ be a linearisation of the considered action and assume it is renormalized such that $t_1+\ldots +t_n=2$.

A linearisation $\mathcal{O} (\mathcal{L}) $ is said to be typical if all semi-stable points are stable and atypical otherwise. In other words, it  is typical exactly when $t_{i_1}+\ldots + t_{i_r} \neq 1$ for any $\{i_1, \ldots i_r\}\subset \{1, \ldots , n\}$. 

Let  $\mathcal{G}(\mathcal{L})$ denotes the  GIT - quotient of $(\C P^{1})^{n}$  by the diagonal $PGL_{2}(\C)$ - action defined by the linearisation  $\mathcal{O} (\mathcal{L})$.

\subsubsection{Domain of admissible weights and GIT-quotients}

Hassett in ~\cite{H}  related  GIT - quotients to  the moduli spaces  of weighted stable genus zero curves.

The domain of admissible weights for weighted stable  genus zero curves is  
\[
\mathcal{D}_{0,n}= \{(t_1,\ldots , t_n)\in \R ^{n} | 0<t_i\leq 1, \;\; t_1+\ldots +t_n>2\}.
\]
It is defined in~\cite{H}  the boundary, not topological,  of $\mathcal{D}_{0,n}$ by
\[
\partial \mathcal{D}_{0,n} = \{ (t_1,\ldots , t_n) | t_1+\ldots +t_n=2, \; 0<t_i<1 \; \text{for}\; i=1,\ldots n\}.
\]

The  description of the chamber structure of $E_{\Q}^{G}(X)\subset f^{-1}_{\Q}(A(X))$  applied to $X=(\C P^{1})^{n}$ and the diagonal action of the group $G=PGL_{2}(\C)$ together with the reformulation of stability conditions,  enabled Hassett to relate moduli spaces of weighted stable genus zero curves to GIT-quotients.

\begin{thm}\label{typical}
Let $\mathcal{L}$ be a typical linearisation in $\partial \mathcal {D}_{0,n}$. There exists an open neighborhood $U$ of $\mathcal{L}$ such that $U\cap \mathcal{D}_{0,n}$ is contained in an open fine chamber of $\mathcal{D}_{0,n}$. For each set of weight data $\mathcal{A}\in U \cap \mathcal{D}_{0,n}$, there is a natural isomorphism
\[
\overline{\mathcal{M}}_{0, \mathcal{A}} \stackrel{\cong}{\to} \mathcal{G}(\mathcal{L}).
\]
\end{thm}
 
For an atypical  linearisation $\mathcal{L}$ of the boundary  $\partial \mathcal {D}_{0,n}$ the description is more complicated. 

\begin{thm}\label{atypical}
Let $\mathcal{L} \in \partial \mathcal{D}_{0,n}$ be an atypical linearisation. Suppose that $\mathcal{L}$ is in the closure of the coarse chamber associated with the weighted data $\mathcal{A}$. Then,  there exists  a natural birational morphism
\[
\rho : \overline{\mathcal{M}}_{0, \mathcal{A}} \to \mathcal{G}(\mathcal{L}).
\] 
\end{thm}

In what follows we denote by $\mathcal{A}(\mathcal{L})$ any weight data associated, according to Theorem~\ref{typical} or Theorem~\ref{atypical}, to $\mathcal{L}=(t_1, \ldots , t_n)\in \partial \mathcal {D}_{0,n}$.
\begin{ex}
In the paper~\cite{H} it is showed that the space 
$(\C P^{1})^{n}$ can be interpreted as $\overline{\mathcal{M}}_{0,\mathcal{A}}$, where $\mathcal{A} = (a_1,\ldots , a_n)$ such that
\[
a_{i_1}+a_{i_2}>1 \; \text{for}\; \{i_1,i_2\}\subset \{1,2,3\};
\]
\[
a_i+a_{j_1}+\ldots +a_{j_r}\leq 1\; \text{for}\; i=1, 2, 3\; \text{and} \; \{j_1, \ldots , j_r\}\subset \{4,\ldots , n\}, \; \text{where}\; r>0.
\]
In addition, if $Q$ is the moduli space of curves of genus zero  with  $n$ distinct points, then $Q$ can be embedded in $(\C P^{1})^{n-3}$ and the compactification of $Q$ by $(\C P^{1})^{n-3}$ is isomorphic to $\overline{\mathcal{M}}_{0, \mathcal{A}(n)}$  where $\mathcal{A}(n) = (a,a,a, \varepsilon, \ldots, \varepsilon)$ and 
$\frac{1}{2} < a\leq 1$, $0<(n-3)\varepsilon \leq 1-a$.   In~\cite{H}, it is also described   the Keel's approach to $\overline{\mathcal{M}}_{0,n}$ in terms of reduction morphisms, precisely it is described the map $\overline{\mathcal{M}}_{0,n} \to (\C P^{1})^{n}$ as a product of reduction morphisms.  
\end{ex}

\section{Equivariant model for the orbit space $G_{n,2}/T^n$}
In our recent papers we constructed  the model for the orbit space $G_{n,2}/T^n$ in terms of the chamber decomposition $\{C_{\omega}\}$ of $\Delta _{n,2}$  by the admissible polytopes, the spaces of parameters $\{F_{\omega}\}$ of the chambers, universal space of parameters $\mathcal{F}_{n}$ and the projections $p_{\omega} : \mathcal{F}_{n}\to F_{\omega}$.

We shortly describe it in more detail. The strata $\{W_{\sigma}\}$  on a Grassmann manifold $G_{n,2}$ are defined as  non-empty sets of the form  
\[
W_{\sigma} = \{ L\in G_{n,2} | P^{I}(L) \neq 0 \; \text {iff} \; I \in \sigma\}, 
\]
where $\sigma$ is a subset of the set of two-elements subsets of $\{1,\ldots , n\}$ and $P^{I}(L)$ are the standard Pl\"ucker coordinates for $L$.

Let $\mu : G_{n,2}\to \Delta _{n,2}$ be the standard moment map defined by
\[
\mu (L) = \frac{1}{\sum_{I} |P^{I}(L)|^2}\sum _{I}P^{I}(L)\Lambda _{I},
\]
where $\Lambda _{I}\in \Z ^{n}$ such that $\Lambda _{I}(i)=1$ iff $i\in I$ and otherwise it is zero.

The following is proved to hold~\cite{BT}:
\begin{itemize}
\item $W_{\sigma}$ is $(\C ^{\ast})^{n}$-invariant,
\item $\mu (W_{\sigma}) = \stackrel{\circ}{P}_{\sigma}$,  where $P_{\sigma}=  \text{convhull}(\Lambda_{I}, I\in \sigma)$,  we call $P_{\sigma}$ an  admissible polytope,
\item $W_{\sigma}/T^n\cong \stackrel{\circ}{P}_{\sigma}\times F_{\sigma}$, where $F_{\sigma} = W_{\sigma}/(\C ^{\ast})^{n}$  is an algebraic manifold.
\end{itemize}

In particular, for the  main stratum $W_n$  consisting  of  points whose all Pl\"ucker coordinates are non-zero, that is 
\[
W_n= \{L \in G_{n,2} | P^{I}(L)\neq 0 \; \text{for any} \; I \subset \{1,\ldots ,n\}, \; |I|=2\}, 
\]
it holds 
\begin{equation}\label{mainstr}
W_n/T^n\cong \stackrel{\circ}{\Delta}_{n,2}\times F_{n},
\end{equation}
 where $F_{n}\subset (\C P^{1})^{N}$, $N={n-2\choose 2}$ is given by
\begin{equation}\label{jednacineFn}
F_{n}= \{((c_{ij}:c_{ij}^{'}))_{3\leq i<j\leq n} \in (\C P^{1}_{A})^{N} | c_{ij}c_{ik}^{'}c_{jk} = c_{ij}^{'}c_{ik}c_{jk}^{'}\},
\end{equation}
and  $\C P^{1}_{A} = \C P^{1}\setminus \{ (1:0), (0:1), (1:1)\}$.

In addition, the main stratum is, in the local coordinates $z_3, \ldots, z_n$, $w_3, \ldots, w_n$ of a Pl\"ucker chart, given  by
\begin{equation}\label{eqmain}
c_{ij}z_iw_j = c_{ij}^{'}z_jw_i, \;\; 3\leq i<j\leq n.
\end{equation}
It is   an  everywhere dense set in $G_{n,2}$, which implies that $W_n/T^n$  is everywhere dense set in $G_{n,2}/T^n$, that is there exists a compactification of $W_n/T^n$, which gives the orbit space $G_{n,2}/T^n$. Together with~\eqref{mainstr}, it implies that there exists compactification of $\stackrel{\circ}{\Delta}_{n,2}\times F_n$ which gives $G_{n,2}/T^n$. We constructed the model for this compactification of the form $\Delta _{n,2}\times \mathcal{F}_{n}$ for the corresponding compactification $\mathcal{F}_{n}$ for $F_{n}$.  The compactification   $\mathcal{F}_{n}$ for $F_n$  is a smooth compact manifold.  Note that $W_{\sigma}/T^n$ belongs to the boundary of  $W_n/T^n\cong \stackrel{\circ}{\Delta}_{n,2}\times F_{n}$ in $G_{n,2}/T^n$. When considering  $W_{\sigma}/T^n$ as a part of boundary of $W_{n}/T^n$ in  $\Delta _{n,2}\times \mathcal{F}_{n}$  we obtain the space  $\tilde{F}_{\sigma}\subset \mathcal{F}_{n}$.  We call  $\tilde{F}_{\sigma}$   a  virtual space of parameters for a stratum $W_{\sigma}$. Then the following holds, see~\cite{BT0},~\cite{BT2nk}:
\begin{itemize}
\item $\tilde{F}_{n} = F_n$,\; $\mathcal{F}_{n}=  \cup _{\sigma} \tilde{F}_{\sigma}$
\item there exists a projection $p_{\sigma} : \tilde{F}_{\sigma}\to F_{\sigma}$ for any $\sigma$,
\item for a point $L\in W_{\sigma}$ we say to be singular if  the  real  space of parameters $F_{\sigma}$ is not homeomorphic to the virtual space of parameters $\tilde{F}_{\sigma}$ .
\end{itemize}

We recall the description  from~\cite{BT2} of the universal space of parameters $\mathcal{F}_{n}$. Let $\bar{F}_{n}$ be the compactification of $F_n$ in $(\C P^{1})^{N}$, that is 
\begin{equation}\label{jednacinebarFn}
\overline{F}_{n} = \{((c_{ij}:c_{ij}^{'}))_{3\leq i<j\leq n} \in (\C P^{1})^{N} | c_{ij}c_{ik}^{'}c_{jk} = c_{ij}^{'}c_{ik}c_{jk}^{'}\}.
\end{equation}
It is a smooth compact manifold.
Let
\[
\hat{F}_I = \{ \bar{F}_n\cap  \{(c_{ik} : c_{ik}^{'}) = (c_{il} : c_{il}^{'}) = (c_{kl} : c_{kl}^{'}) = (1 : 1)\},
\]
for $I = \{i, k, l\} \in \{I \subset \{1,\ldots  , n\}, |I| = 3\}$ and  $n \geq  5$.

Let the set  $\mathcal{G}_{n}$ be defined by  
\begin{itemize}
\item  $\mathcal{G}_{n} = \emptyset$  for $n = 4$,
\item  $\mathcal{G}_{n}$ is the set  of  all possible non-empty intersection of $\hat{F}_{I}$'s, that is $F_{I}=\hat{F}_{I_1}\cap\ldots \cap \hat{F}_{I_k}$, where $I=\{I_{i_1}, \ldots , I_{i_k}\}$,  for $n\geq 5$.
\end{itemize}

Recall that De Concini and Processi~\cite{DCP}, Fulton and MacPherson~\cite{FMP},  Li~\cite{LILI} introduced the notion of wonderful compactification of a smooth algebraic variety generating by a building set of its smooth subvarieties.   

The smooth, compact manifold  $\mathcal{F}_{n}$    which is the wonderful compactification of the smooth variety $\bar{F}_{n}$  with the building set $\mathcal{G}_{n}$ is proved in~\cite{BT2} to satisfy requirements for the description of the model for  $G_{n,2}/T^n$. We call it an universal space of parameters.

As far as the admissible polytopes are concerned, it  is proved in~\cite{BT1} that the admissible polytopes $P_{\sigma}$ such that $\stackrel{\circ}{P}_{\sigma}\subset \stackrel{\circ}{\Delta}_{n,2}$ may have the  dimension $n-2$ or $n-1$.  They are explicitly described in~\cite{BT1} as follows.

\begin{thm}
The admissible polytopes  $P_{\sigma}\subset \stackrel{\circ}{\Delta}_{n,2}$   for $G_{n,2}$ of dimension $n-2$ are  given by the intersection of $\Delta _{n,2}$ with the planes 
\[
\sum_{i\in S, |S|=p}x_i=1, \;  \text{where}\; S\subset \{1,\ldots ,n\}, \; 2\leq p\leq [\frac{n}{2}].
\] 
\end{thm}

\begin{thm}\label{dimensionn-1}
The admissible polytopes  $P_{\sigma}\subset \stackrel{\circ}{\Delta}_{n,2}$   for $G_{n,2}$ of dimension $n-1$ are  given  by $\Delta _{n,2}$ and intersections with $\Delta _{n,2}$ of all collections of the half-spaces of the form
\[
H_{S} : \sum _{i\in S} x_i<1, \; S\subset \{1,\ldots ,n\}, \; |S|=k, \; 2\leq k\leq n-2,
\]
such that if $H_{S_1}, H_{S_2}$ belong to a collection then $S_1\cap S_2=\emptyset$.
\end{thm}

The admissible polytopes define the chamber decomposition of $\Delta _{n,2}$ where the chambers $C_{\omega}$ are given by
\begin{equation}\label{chdef}
C_{\omega} = \bigcap\limits_{\sigma\in \omega} \stackrel{\circ}{P}_{\sigma} \neq \emptyset \;\; 
 \text{and} \;\;  C_{\omega}\cap \stackrel{\circ}{P}_{\sigma}=\emptyset \; \text{for}\; \sigma \not \in \omega.
\end{equation}

According to~\cite{BT1} this chamber decomposition can be obtained by the hyperplane arrangement that consists of
the planes $\Pi _{n}$ given by the equations  $\sum _{i\in S, \|S\| =p}x_i=1$, where $S\subset \{1,\ldots, n\}$, $2\leq p\leq [ \frac{n}{2}]$.  Now, consider a hyperplane arrangement in $\R ^{n-1} =
\{(x_1,\ldots , x_n) \in \R ^n | x_1+\ldots +x_n=2\}$  given by
\begin{equation}\label{hyparrang}
\mathcal{R}_{n} = \Pi _{n} \cup  \{x_i = 0, 1 \leq i  \leq  n\} \cup  \{ x_i = 1, 1 \leq  i \leq  n\}
\end{equation}
and the face lattice $L(\mathcal{R}_{n})$ of this arrangement. The intersection $L(\mathcal{R}_{n,2}) = L(\mathcal{R}_{n}) \cap  \Delta _{n, 2}$
produces  the  decomposition of $\Delta _{n,2}$ , which by~\cite{BT1} coincides  with the decomposition~\eqref{chdef}.

Further,  let $\hat{C}_{\omega} = \hat{\mu}^{-1}(C_{\omega})\subset G_{n,2}/T^n$, where $\hat{\mu} : G_{n,2}/T^n\to \Delta _{n,2}$ is induced by the moment map. The following holds~\cite{BT1}:

\begin{prop}\label{ComegaF}
For any $C_{\omega}$ there exists the  canonical homeomorphism
\[
h_{\omega} : \hat{C}_{\omega} \to C_{\omega}\times F_{\omega},
\]
where $F_{\omega} = \cup _{\sigma\in \omega}F_{\sigma}$  is a compactification of  $F_{n}$ and $F_{\omega}$ is a smooth manifold  for $\dim C_{\omega}=n-1$.
\end{prop}

The virtual spaces of parameters, which correspond to the  chambers are proved to behave nicely regarding an universal space of parameters~\cite{BT1}:

\begin{thm}\label{univ-virt}
For any $C_{\omega}\subset \stackrel{\circ}{\Delta}_{n,2}$ it holds
\[
\bigcup\limits_{C_{\omega}\subset \stackrel{\circ}{P}_{\sigma}}\tilde{F}_{\sigma} = \mathcal{F}_{n}.
\]
Moreover, this union is disjoint, which implies that for any chamber $C_{\omega}$ it is defined the projection $p_{\omega} : \mathcal{F}_{n}\to F_{\omega}$ by $p_{\omega}(y) = p_{\sigma}(y)$ where $y\in \tilde{F}_{\sigma}$.
\end{thm}

The previous statements altogether  make us possible to  construct  a model for the orbit space $G_{n,2}/T^n$.  For that purpose  we first recall that the preimage
$\hat{\mu}^{-1}(\partial \Delta _{n,2})$ consists of $n$ copies of $G_{n-1,2}/T^{n-1}$ and $n$ copies of $\C P^{n-2}$. 

The universal space of parameters for $\C P^{n-2}$ is a point,  while for the universal space of parameters for $G_{n-1,2}(q)/T^{n-1}\subset G_{n,2}/T^n$, $1\leq q\leq n$ the following holds~\cite{BT1}:
\begin{prop}\label{restr}
The universal space of parameters $\mathcal{F}_{n-1,q}$ for $G_{n-1,2}(q)$, $1\leq q\leq n$ can be obtained from the universal space of parameters $\mathcal{F}_{n}$ for $G_{n,2}$  by the restriction 
\[
\mathcal{F}_{n-1,q} = \mathcal{F}_{n}|_{(c_{ij}:c_{ij}^{'}), i,j\neq q},
\]
which defines the projection $r_{q} : \mathcal{F}_{n} \to \mathcal{F}_{n-1,q}$.
\end{prop}

Proposition~\ref{restr}  gives that Theorem~\ref{univ-virt} can be extended to an an arbitrary chamber $C_{\omega}$.

\begin{prop}\label{projanych}
For any chamber $C_{\omega}\subset \partial \Delta _{n,2}$ it is defined the  projection $p_{\omega} : \mathcal{F}_{n}\to F_{\omega}$. If  $C_{\omega}\subset \Delta ^{n-2}(q)$ this projection maps $\mathcal{F}_{n}$ to a point, while for $C_{\omega}\subset \Delta _{n-1,2}(q)$ it is defined by $p_{\omega} = p_{\omega}(q)\circ r_{q}$, where $p_{\omega}(q) : \mathcal{F}_{n-1,q}\to F_{\omega}$ and $1\leq q\leq n$.
\end{prop}

The model $(U_n, G_n)$  for the orbit space $G_{n,2}/T^n$ is then given by
\[
U_{n} = \Delta _{n,2}\times \mathcal{F}_{n},
\]
and the projection $G_{n} : U_{n}\to G_{n, 2}/ T^n = \cup \hat{C}_{\omega}$ is  defined by 
\begin{equation}\label{projection}
G_{n}(x, y) = h_{\omega}^{-1}(x, p_{\omega}(y) \; \text{for}\; x\in C_{\omega}.
\end{equation}
It is proved in~\cite{BT1}:
\begin{thm}
The map $G_{n}$ is correctly defined, it is a continuous surjection and the orbit space $G_{n,2}/T^n$ is homeomorphic to the quotient of the space $U_{n}$ by the map $G_n$.
\end{thm}

\begin{rem}
The space $U_{n}$  functorially resolves the singular points of the  orbit space $G_{n,2}/T^n$ and it is  compatible with the combinatorial structure of  $G_{n,2}/T^n$ previously described .
This means the following. Let $Y_n = (G_{n,2}/T^n )\setminus  (\text{Sing}(G_{n,2}/T^n))$, where the singular points are defined as above,  and note that $Y_n$  is an open, dense set in $G_{n,2}/T^n$,   that is it  is a
manifold. Then,   there exists a projection $G_n : U_{n}\to G_{n,2}/T^n$  such that for an open, dense
manifold $V_n = G_n^{-1}(Y_n) \subset U_{n}$  the map $G_n : V_{n} \to Y_n$  is a diffeomorphism. This projection  $G_{n}: U_{n}\to G_{n,2}/T^n$ preserves the  stratification of $U_{n}\setminus V_n$ given by $C_{\omega}\times \tilde{F}_{\sigma}$, $\sigma \in \omega$  such that $\tilde{F}_{\sigma}\not\cong F_{\sigma}$ and the stratification of $(G_{n,2}/T^n) \setminus Y_n$ given by $W_{\sigma}/T^n$ such that  $\tilde{F}_{\sigma}\ncong F_{\sigma}$.
\end{rem}

\section{ The model for $G_{n,2}/T^n$  in terms of  moduli spaces of  weighted pointed stable  genus zero curves}

The structural data for our model $(U_n, G_n)$ for the orbit space $G_{n,2}/T^n$  can be summarized as follows:
\begin{itemize}
\item The universal space of parameters $\mathcal{F}_{n}$;
\item The virtual spaces of parameters $\tilde{F}_{\sigma}$;
\item The spaces of parameters $F_{\sigma}$ of the strata and  spaces of parameters  $F_{\omega}$ over the chambers;
\item The projections $p_{\sigma} : \tilde{F}_{\sigma}\to F_{\sigma}$, $p_{\omega} : \mathcal{F}_{n}\to F_{\omega}$ and $r_{q}: \mathcal{F}_{n}\to \mathcal{F}_{n-1, q}$, where  $1\leq q\leq n$.
\end{itemize}

We show that  these data  can be interpreted in terms of moduli spaces of weighted stable genus zero curves and the  morphisms between them.

We first recall  the well known fact that the space of parameters $F_{n}$ given by~\eqref{jednacineFn} is isomorphic to $\mathcal{M}_{0,n}$, which is the moduli space of $n$-pointed  genus zero curves.  More precisely, from~\eqref{jednacineFn} we see that 
\begin{equation}\label{CPAdiagon}
\mathcal{M}_{0,n} \cong F_{n} \cong (\C P^{1}_{A})^{n-3}\setminus \Delta _{n-3}, 
\end{equation}
for the complete diagonal $\Delta _{n-3}$, see also~\cite{BT0}.

We recall at this point,  that it is given in~\cite{lando} an inductive description of $\mathcal{M}_{0, n}$ using the forgetting bundle $\pi : \mathcal{M}_{0, n+1}\to \mathcal{M}_{0, n}$, whose fiber  is $\C P^1$ having removed $n$ points.  We reformulate this result in terms of representation~\eqref{CPAdiagon}.

\begin{lem}
The projection map $\pi :  (\C P^{1}_{A})^{n-2}\setminus \Delta _{n-2} \to  (\C P^{1}_{A})^{n-3}\setminus \Delta _{n-3}$ defined by  forgetting the last coordinate is a fiber bundle whose fiber is $\C P^{n}$ with removed $n$ points.
\end{lem}

\begin{proof}
The preimage  by $\pi$ of any  point $((c_1:c_1^{'}), \ldots , (c_{n-3}:c_{n-3}^{'}))\in (\C P^{1}_{A})^{n-3}\setminus \Delta _{n-3}$ is $((c_1:c_1^{'}), \ldots , (c_{n-3}:c_{n-3}^{'}), (c_{n-2}: c_{n-2}^{'}))$ such that $(c_{n-2}:c_{n-2}^{'}) \neq (1:0), (0:1), (1:1)$ and $(c_{n-2}: c_{n-2}^{'})\neq (c_i :c_{i}^{'})$ for $1\leq i\leq n-3$. Thus, $\pi$ is a fiber bundle with the fiber $\C P^{1}$ having removed $n$ points.
\end{proof}

\begin{lem}\label{paramspacestratum}
The space of parameters $F_{\sigma}$ of a stratum $W_{\sigma}\neq W_n$ is either a point either it is homeomorphic  to $F_{m}$, where  $\ m\leq n-1$.
\end{lem}

\begin{proof}

We can assume that the Pl\"ucker coordinate $P^{12}\neq 0$ on $W_{\sigma}$. Then $W_{\sigma}$ is in local coordinates given by the points $(z_{3}, \ldots, z_{n}, w_{3}, \ldots ,w_{n})$ where $z_{i} = P^{2i}(L)$ and $w_{i}=P^{1i}(l)$, $3\leq i\leq n$, for a point $L\in W_{\sigma}$, that is $W_{\sigma} \subset \C ^{n-2}_{z} \times \C ^{n-2}_{w}$.
From the definition of a stratum, it follows that  a coordinate of a  point of a stratum $W_{\sigma}$ is non-zero if and only if it is non-zero coordinate for all points of the stratum $W_{\sigma}$. Thus, let   $N_z =\{ k  | z_{k+2}= 0\}$ and  $N_w =\{ k  | w_{k+2}= 0\}$ Then,   $W_{\sigma}\subset \C^{L_z}\times C^{L_{w}} \subset \C^{n-2}_{z}\times \C ^{n-2}_{w}$ where $L_z= \{1,\ldots , n-2\}\setminus N_z$ and  $L_w= \{1,\ldots ,n-2\}\setminus N_w$,  The strata $W_{\sigma}$ is  given in   $\C ^{L_z}\times C^{L_{w}}$  by the system of equations
\begin{equation}\label{str}
c_{ij}z_{i}w_{j} = c_{ij}^{'}z_{j}w_{i}, \; \text{where}\;  z_i, w_j\neq 0
\end{equation}
and $(c_{ij} : c_{ij}^{'} )\in \C P^{1}\setminus \{(1:0), (0:1)\}$.  In addition the parameters $(c_{ij} : c_{ij}^{'})$ satisfy the system of equation 
\begin{equation}\label{param}
c_{ij}c_{ik}^{'}c_{jk}= c_{ij}^{'}c_{ik}c_{jk}^{'}.
\end{equation}
Note that if $(c_{ij}:c_{ij}^{'}) = (1:1)$ for some $i,j$ then the coordinate $w_{j}$ can be removed.  If we do that with all such $i,j$ we obtain the  new set $\tilde{L}_{w}\subset L_{w}$ and  $W_{\sigma}\subset \C^{L_z}\times  \C ^{\tilde{L}_{w}}$.Then $W_{\sigma}$  is given in  $\C^{L_z}\times  \C ^{\tilde{L}_{w}}$  by the system~\eqref{str}, where $(c_{ij}:c_{ij}^{'})\in \C P^{1}\setminus \{(1:0), (0:1), (1:1)\}$, subject to the equations~\eqref{param}.  It implies that $F_{\sigma} \cong F_{m}$ for some $4\leq m\leq n-1$.
\end{proof}

Using Keel's description of $\overline{\mathcal{M}}_{0,n}$  from~\cite{K},   it is proved in~\cite{BT2}:

\begin{thm}\label{Keel-univ}
The universal space of parameters $\mathcal{F}_{n}$   is diffeomorphic to $\overline{\mathcal{M}}_{0,n}$.
\end{thm}

\begin{rem}
As far as the real spaces of parameters $F_{\omega}$ of the chambers are concerned we recall from~\cite{BT2} that the proof of Proposition~\ref{ComegaF} relies on the observation from~\cite{GMP}, which is  based on~\cite{AT},  that for a given  $C_{\omega}$ the preimage $\hat{\mu}^{-1}(x) = \mu ^{-1}(x)/T^n \subset G_{n,2}/T^n$ does not depend on a point $x\in C_{\omega}$ and we denote it by $F_{\omega}$.  Moreover, the quotient  $\mu ^{-1}(x)/T^n$, and thus $F_{\omega}$,   can be identified with:
\begin{itemize}
\item categorical quotient of $W_{\omega} = \cup _{\sigma \in \omega}W_{\sigma}$ by the action of $(\C ^{\ast})^{n}$, following Mumford~\cite{M},
\item  GIT quotient $\mathcal{G}(\mathcal{L})$ of $(\C P^{1})^{n}$ by the diagonal $PGL_{2}(\C)$ -action defined by the linearisation $\mathcal{O}(\mathcal{L})$ for $\mathcal{L} = (t_1, \ldots, t_n)$ for an arbitrary $(t_1, \ldots ,t_n)\in C_{\omega}$, following Kirwan~\cite{K}
\end{itemize}
\end{rem}
Therefore we deduce:
\begin{thm}\label{omegagit}
The space of parameters $F_{\omega}$ of any chamber $C_{\omega}\subset \Delta _{n,2}$ is homeomorphic to the GIT quotient  $\mathcal{G}(\mathcal{L})$ of $(\C P^{1})^{n}$ defined by  the linearisation $\mathcal{O}(\mathcal{L})$ for an arbitrary $\mathcal{L} = (t_1,\ldots , t_n) \in C_{\omega}$.
\end{thm}
Together with Theorem~\ref{typical}  it  gives:
\begin{thm}\label{omegaweight}
Let  $C_{\omega}\subset \stackrel{\circ}{\Delta} _{n,2}$ such that $\dim C_{\omega}=n-1$. Then,  there exists neighborhood $U$  of   a point $(t_1,\ldots ,t_n)\in C_{\omega}$ such that for any  $\mathcal{A} = (a_1,\ldots , a_n)\in U\cap \mathcal{D}_{0,n}$ there is the natural isomorphism
\[
\overline{\mathcal{M}}_{0, \mathcal{A}} \stackrel{\cong}{\to} F_{\omega}.
\]
\end{thm}

\begin{thm}\label{omegabirat}
 Let  $C_{\omega}\subset \stackrel{\circ}{\Delta} _{n,2}$ such that   $C_{\omega}\subset \stackrel{\circ}{P}_{\sigma}$ for some   $P_{\sigma}$, $\dim P_{\sigma}=n-2$. For a $(t_1, \ldots , t_n)\in C_{\omega}$, let $\mathcal{A }=(a_1,\ldots, a_n)\in \mathcal{D}_{0,n}$  is  such that $(t_1, \ldots ,t_n)$ is in the closure of the  coarse chamber for $\mathcal{D}_{0, n}$  containing  $\mathcal{A}$.  Then there exists natural birational morphism
\[
\hat{\rho} : \overline{\mathcal{M}}_{0, \mathcal{A}} \to F_{\omega}.
\]

\end{thm}

Denote by $\mathcal{C}_{n,2}$ the family of chambers from $\stackrel{\circ}{\Delta}_{n,2}$,  by $\mathcal{C}_{\mathcal{H}}$ the family of  chambers from  the fine  chamber decomposition $\mathcal{W}_{f}$ for $\mathcal{D}_{0,n}$ and by $S(\mathcal{C}_{\mathcal{H}})$ the set of all its subsets.   From Theorem~\ref{omegaweight} and Theorem~\ref{omegabirat} we deduce:

\begin{enumerate}
\item For $C_{\omega} \in \mathcal{C}_{n,2}$ such that $\dim C_{\omega}=n-1$ there exists unique $D_{\omega} \in \mathcal{C}_{\mathcal{H}}$ such that $C_{\omega} \subset \overline{D}_{\omega}$.
\item For $C_{\omega}\in \mathcal{C}_{n,2}$ such that $\dim C_{\omega}\leq n-2$, there exists  $D_{\omega}\in \mathcal{C}_{\mathcal{H}}$ such that $C_{\omega}\subset \overline{D}_{\omega}$ and such $D_{\omega}$ is not  unique in general. Therefore, it arises  the set $\{ D_{\omega}\in \mathcal{C}_{\mathcal{H}} | C_{\omega}\subset \overline{D}_{\omega}\}\in S(\mathcal{C}_{\mathcal{H}}) $.
\end{enumerate}

\begin{ex} To verify the  non-uniqueness of $D_{\omega}$ from the second case  we take $n=5$ and the chamber $C_{\omega}\in \mathcal{C}_{5,2}$ defined by:
\[
C_{\omega} : x_1+x_3=1, \;\; x_i+ x_j <1, \; \text{for}\; 1\leq i<j\leq 5, \; \{i, j\}\neq \{1,3\}, \;\; x_1+\ldots +x_5=2.
\]
This chamber is well defined, that is non-empty, say $(\frac{3}{5}, \frac{1}{3}, \frac{2}{5}, \frac{1}{3}, \frac{1}{3})\in C_{\omega}$ and $\dim C_{\omega}=3$.

Consider the fine chamber $\tilde{C}_{\omega}$ in $\R ^{5}$ defined by
\[
\tilde{C}_{\omega}  : x_1+x_3 >1, \;\; x_i+ x_j <1, \; \text{for}\; 1\leq i<j\leq 5, \; \{i, j\}\neq \{1,3\},
\]
and the chamber $D_{\omega}\in \mathcal{C}_{\mathcal{H}}$ given by $D_{\omega} = \tilde{C}_{\omega} \cap \mathcal{D}_{0, 5}$.
Let $D_{\omega _1}\in \mathcal{C}_{\mathcal{H}}$ be a chamber  given by the facet of $D_{\omega}$ defined by the hyperplane $x_1+x_3=1$.   It is obvious that $C_{\omega} \subset  \overline{D}_{\omega}$ and $C_{\omega}\subset \overline{D}_{\omega _1}$.
\end{ex}

Altogether:

\begin{lem}
There exists injective map $\xi : \mathcal{C}_{n,2}\to S(\mathcal{C}_{\mathcal{H}})$ defined by
\[
\xi (C_{\omega})  =\{ D_{\omega}, \;\; C_{\omega}\subset \overline{D}_{\omega}\},
\]
where  $|\xi (C_{\omega})|=1$ if and only if $\dim C_{\omega}=n-1$.
\end{lem}
We also  note the following:
\begin{lem}
If a chamber $C_{\omega}\in \mathcal{C}_{n,2}$  is in the intersection of maximum $k$ hyperplanes from the arrangement~\eqref{hyparrang},  then there are exactly  $2^{k}$ chambers in $\mathcal{C}_{\mathcal{H}}$ having $C_{\omega}$ as a facet.
\end{lem}
\begin{proof}
Assume that $k=1$ that is $C_{\omega}$ belongs to  the hyperplane $x_1+x_2=1$ meaning that it  is  given by the intersection of  $\Delta _{n,2}$ with this hyperplane and the half spaces defined  by the other hyperplanes from~\eqref{hyparrang} . Since the  intersection of this  hyperplane with $\Delta _{n,2}$ is the same as the intersection of the hyperplane  $x_3+\ldots +x_n=1$ with  $\Delta _{n,2}$, we see that if in the above description we replace $x_1+x_2=1$ with $x_3+\ldots +x_n=1$ we obtain the same chamber $C_{\omega}$. 
Now, if we omit  the hyperplane $x_1+\ldots +x_n=2$ which defines the hypersimplex,  and keep all halfspaces, we obtain two  different chambers $D_{\omega _1}$ and $D_{\omega _2}$ in $\mathcal{D}_{0, n}$ defined by $x_1+x_2=1$ and $x_3+\ldots +x_n=1$,  respectively. Obviously,  $\dim D_{\omega_1}=\dim D_{\omega _2}=\dim C_{\omega} +1$ and $C_{\omega}$ is a facet for both of them, and there are no others such chambers.
Applying the same argument  the statement follows for an arbitrary $k$.
\end{proof}

We denote by $\hat{\mathcal{C}}_{\mathcal{H}} = \xi (\mathcal{C}_{n,2})$.   Theorem~\ref{omegaweight} and Theorem~\ref{omegabirat} imply  that 
 any $\overline{\mathcal{M}}_{0, \mathcal{A}}$ such that $\mathcal{A}\in D_{\omega}\in \xi(C_{\omega})\in \hat{\mathcal{C}}_{\mathcal{H}}$ for some $C_{\omega}\in \mathcal{C}_{n,2}$  can be embedded up to birational morphism in  the orbit space $G_{n,2}/T^n$. More precisely:

\begin{cor}\label{embedhassett}
Let $C_{\omega}\subset \stackrel{\circ}{\Delta} _{n,2}$. Then
\begin{enumerate}
\item if $\dim C_{\omega}=n-1$, the manifold $\overline{\mathcal{M}}_{0, \mathcal{A}}$ can be embedded  into $G_{n,2}/T^n$, for a weight vector  $\mathcal{A}$ satisfying the conditions of  Theorem~\ref{omegaweight}.
\item if $\dim C_{\omega}< n-1$, the space $\overline{\mathcal{M}}_{0, \mathcal{A}}$  can be embedded into $G_{n,2}/T^n$ up to birational morphism, for a weight vector  $\mathcal{A}$ satisfying the conditions of Theorem~\ref{omegabirat}.
\end{enumerate}
\end{cor}

\begin{rem}
Since the space $G_{n,2}/T^n$ is in general far from being a manifold we explain  the meaning of the phrase that a manifold $\overline{\mathcal{M}}_{0,\mathcal{A}} $  can be embedded in this space.  It means that the diagram
\begin{equation}\label{diag}
\begin{CD}
 \mathcal{F}_{n} @>{f}>>\Delta_{n,2}\times \mathcal{F}_{n}\\
@VV \rho _{n,\mathcal{A}} V @VVG_{n} V \\ 
 \overline{\mathcal{M}}_{0, \mathcal{A}}@>{g}>>G_{n,2}/T^n
\end{CD}\end{equation}
is commutative, where
\begin{itemize}  
\item $f: \mathcal{F}_{n}\to \Delta _{n,2}\times \mathcal{F}_{n}$ is defined by $f(p) = (x, p)$ for a fixed point $x\in C_{\omega}$, 
\item $\mathcal{A} \in \xi (C_{\omega})$ and $F_{\omega}\cong \overline{\mathcal{M}}_{0, \mathcal{A}}$ and  $\rho _{n, \mathcal{A}}$ is a reduction morphism,
\item  $G_n$ is the projection defined by~\eqref{projection}  and $g$ is defined by the commutativity of the diagram.
\end{itemize}
\end{rem}  



\subsection{Morphisms and projection maps}

The  exceptional locus of the projection $p_{\omega}$ incorporated in our model is as follows:

\begin{thm}\label{exceptlocus}
The exceptional locus of the projection $p_{\omega} : \mathcal{F}_{n} \to F_{\omega}$ for  $C_{\omega}\subset \stackrel{\circ}{\Delta}_{n,2}$  is given by those $\tilde{F}_{\sigma}$, $\sigma\in \omega$ for which 
$\tilde{F}_{\sigma}$ is not homeomorphic to $F_{\sigma}.$
\end{thm}

\begin{proof}
It follows from Theorem~\ref{univ-virt} that for a chamber $C_{\omega}\subset \stackrel{\circ}{\Delta}_{n,2}$  the outgrows in the compactification of  $F_{\omega}$  to $\mathcal{F}_{n}$, which is   given by $\mathcal{F}_{n}= \cup _{\sigma \in \omega} \tilde{F}_{\sigma}$,  are  those $F_{\sigma}$ such that $\tilde{F}_{\sigma}$ is not homeomorphic to $F_{\sigma}$.  
\end{proof}

\begin{prop}\label{virtual}
A virtual space of parameters  $\tilde{F}_{\sigma}\subset \mathcal{F}_{n}$  is homeomorphic to $\mathcal{F}_{s}\times \mathcal{F}_{l}\times \hat{F}_{q}$, where
\begin{itemize}
\item $\hat{F}_{q} \cong F_{m}$, $m\leq q$ or
\item   $\hat{F}_{q} = \hat{\mathcal{F}}_{q}$   is a wonderful compactification of $\{(c_{ij} :c_{ij}^{'})\in (\C P^{1})^{N_q} | c_{ij}c_{il}^{'}c_{jl} = c_{ij}^{'}c_{il}c_{jl}^{'}, \; c_{ij},  c_{ij}^{'}\neq 0\}$ for  $N_{q} = \binom{q-2}{2}$, with the building set induced from $\mathcal{G}_{n}$.
\end{itemize}
For $s, l, q$ the following holds:  either $0<s , l, q\leq n$ and $s+l+q = n+4$ or one of $s, l, q$ is zero and $s+l+q= n+2$ or two of them are zero and $s+l+q =  n$.   
The projection $p_{\sigma} : \tilde{F}_{\sigma}\to F_{\sigma}$ is given by $p_{\sigma}(f_k, f_l, f_m)=f_m$.

  In addition, 
if $(t_1,\ldots , t_n)\in \stackrel{\circ}{P}_{\sigma}$ then  there exist $\{i_1, \ldots, i_{s-1}\}, \{j_1, \ldots , j_{l-1}\}\subset \{1, \ldots ,n\}$ such that $\{i_1,\ldots , i_{s-1}\}\cap \{j_1, \ldots , j_{l-1}\} = \emptyset$ and $t_{i_1}+\ldots +t_{i_{s-1}}< 1$, $t_{j_1}+\ldots +t_{j_{l-1}}<1$
\end{prop}

\begin{proof}
If $\tilde{F}_{\sigma}\cong F_{\sigma}$,  then $\tilde{F}_{\sigma}\cong F_m$ by Lemma~\ref{paramspacestratum}. Assume that  $\tilde{F}_{\sigma}$ is not homeomorphic to $F_{\sigma}$ and assume that $W_{\sigma}$ belongs to the chart $M_{12}$. Then, there exist local coordinates $z_{i_1}, \ldots , z_{i_{s-2}}=0$ and $w_{j_1}, \ldots , w_{j_{l-2}}=0$, where $\{i_1, \ldots , i_{s-2}\} \cap \{j_1, \ldots , j_{l-2}\} =\emptyset$. Due to the $S_n$-action we may take these coordinates to be  $z_3, \ldots, z_s$ and $w_{s+1}, \ldots , w_{l+s-2}$.
The points  $(c_{ij} : c_{ij}^{'})$  considering to be from  $(\C P^{1})^{N}$, $N\binom{n-2}{2}$ and  which satisfy equations~\eqref{eqmain} for the values of the local coordinates which correspond to   $W_{\sigma}$, are given by  $\bar{F}_{s}\times \bar{F}_{l}\times  F_{q}^{'}$,  where  $F_{q}^{'}$ differs from $F_q$ by the condition that a coordinate $(c_{ij}:c_{ij}^{'})$ is allowed to be $(1:1)$, all subject  to the relations $c_ {ij}c_{ik}^{'}c_{jk} = c_{ij}^{'}c_{ik}c_{jk}^{'}$.  We note that these relations  do not give any relations among the factors $\bar{F}_{s}, \bar{F}_{l}$ and  $F_{q}^{'}$. For example, if $(c_{ij}:c_{ij}^{'})\in \bar{F}_{s}$ then $3\leq i, j\leq s $ and for any $k\geq s+1$ we have that $(c_{ik}:c_{ik}^{'})= (c_{jk}: c_{jk}^{'})= (0:1)$ does not belong $\bar{F}_{l}$  neither $\hat{F}_q$. 

If   the restriction of the building set $\mathcal{G}_{n}$ on $F_{q}^{'}$ is  empty set, then $F_{q}^{'}$  is homeomorphic to some $F_{m}$, where $m\leq q$. Namely, if there exists $(c_{ij}:c_{ij}^{'})\in \hat{F}_{q}$ and $(c_{ij}:c_{ij}^{'})=(1:1)$,  we obtain that $c_{ip}^{'}c_{jp} = c_{ip}c_{jp}^{'}$, which gives 
$(c_{ip}:c_{ip}^{'}) = (c_{jp}:c_{jp}^{'})$,  so one of these two coordinates can be omitted. Thus, in this case we obtain that $\tilde{F}_{\sigma}\subset \mathcal{F}_{n}$ is given by $\mathcal{F}_{s}\times \mathcal{F}_{l}\times F_{m}$.

 Let $\mathcal{G}_{n,q}\neq \emptyset$ be the restriction of the building set $\mathcal{G}_{n}$ on $\overline{F_{q}^{'}} = \bar{F}_{q}\subset (\C P^{1})^{N_q}$. Then $\tilde{F}_{\sigma}\subset \mathcal{F}_{n}$ is given by $\mathcal{F}_{s}\times \mathcal{F}_{l}\times \hat{\mathcal{F}}_{q}$, where   $\hat{\mathcal{F}}_{q}$ is the restriction to $F_{q}^{'}$  of the wonderful compactification for  $\bar{F}_{q}$ with the building set $\mathcal{G}_{n, q}$.


Note that  $\Lambda _{i_{p}i_{q}}$, $1\leq p<q\leq s-2$, then $\Lambda _{2i_{p}}$, $1\leq p\leq s-2$ and $\Lambda _{j_{p}j_{q}}$, $1\leq p<q\leq l-2$, then $\Lambda_{1j_{p}}$, $1\leq p\leq l-2$  are not among  the vertices of the polytope $P_{\sigma}$.  Thus,  following~\cite{BT2} we see that   $P_{\sigma}$ belongs to the halfspaces $x_2+x_{i_1}+\ldots +x_{i_{s-2}}\leq 1$ and $x_1+x_{j_1}+\ldots +x_{j_{l-2}}\leq 2$.  It implies that for 
$(t_1, \ldots , t_n) \in \stackrel{\circ}{P}_{\sigma}$ we have that $t_2+t_{i_1}+\ldots +t_{i_{s-2}}<1$ and  $t_1+t_{j_1}+\ldots +t_{j_{l-2}}< 1$.

\end{proof}

It follows from  Proposition~\ref{virtual} that the closure    $\overline{\tilde{F}}\subset \mathcal{F}_{n}$  is  $\mathcal{F}_{s}\times \mathcal{F}_{l}\times\mathcal{F}_{q}$.



Denote by $\overline{F}_{\sigma,\omega}$ the closure of $F_{\sigma}\cong F_m$ in $F_{\omega}$ for $\sigma \in \omega$. 

\begin{prop}\label{proj}
The following holds:
\begin{enumerate}
\item  $\overline{F}_{\sigma, \omega}$ is the union of $F_{\sigma ^{'}}$ for some $\sigma ^{'} \in \omega$,
\item  the projection $p_{\omega} : \mathcal{F}_{n}\to F_{\omega}$ induces the projection $p_{\omega, \sigma} : \tilde{\mathcal{F}}_{\sigma} \to F_{\sigma, \omega}$.
\item  $\tilde{\mathcal{F}}_{\sigma}$ is the union of the corresponding $\tilde{F}_{\sigma ^{'}}$,
\end{enumerate}
\end{prop}
\begin{proof}
To prove the first statement we note that $C_{\omega}\times  \bar{F}_{\sigma, \omega} \subset \overline{W_{\sigma}/T^n}$, so if $\bar{F}_{\sigma, \omega}\cap F_{\sigma ^{'}}\neq \emptyset$ then $ \overline{W_{\sigma}/T^n}\cap W_{\sigma}^{'}/T^n \neq \emptyset$. For the case of $G_{n,2}$ manifolds it follows from~\cite{GS} that  $W_{\sigma}^{'}/T^n\subset 
\overline{W_{\sigma}/T^n}\cap W_{\sigma}^{'}/T^n $, which implies that $ F_{\sigma ^{'}}\subset \bar{F}_{\sigma, \omega}$.  

 The second statement is obvious and it implies the third statement.
\end{proof}



The dimension of $\mathcal{F}_{n}$ is $2n-6$, so the divisors of maximal dimension  in $\mathcal{F}_{n}$ for the projection $p_{\omega} : \mathcal{F}_{n}\to F_{\omega}$ are given by those $\overline{\tilde{F}}_{\sigma}$ of dimension $2n-8$.  Since by  Lemma~\ref{virtual} for $k, s ,q>0$  we have that $\dim \overline{\tilde{F}}_{\sigma} =  2(k+s+q)-18 \leq 2n-10$, it follows that for a such divisor   $\overline{\tilde{F}}_{\sigma}$  exactly one of $l, s$ is zero in order to have singularity. Thus, we obtain:

\begin{cor}
The divisors of maximal dimension for   the projection $p_{\omega} : \mathcal{F}_{n}\to F_{\omega}$ are given by
\[
\overline{\tilde{F}_{\sigma}} = \mathcal{F}_{p}\times \mathcal{F}_q,
\]
 where  $p, q> 0$ and $ p+q=n+2$ or $q=0$ and $p=n-1$.  The case $q=0$ and $p=n-1$ is possible when $\dim C_{\omega}\leq n-2$ and it corresponds to those  $\sigma\in \omega$ such that $\dim P_{\sigma}=n-2$.
\end{cor}

\begin{cor}
The projection $p_{\omega} : \mathcal{F}_{n}\to F_{\omega}$ contracts the boundary divisors $D_{s, l, q}=\mathcal{F}_{s}\times \mathcal{F}_{l}\times \mathcal{F}_{q}$ by the composition of the projection and the  blowing down
\[ 
D_{s,l,q}\to \mathcal{F}_{q} \to \bar{F}_{q}.
\]
\end{cor}

Together with Theorem~\ref{divisorsFn} which we prove later one, we deduce:

\begin{prop}
Let $\dim C_{\omega}=n-1$. The  space $\mathcal{D}_{s,l,q} = \overline{\mathcal{M}}_{0, s}\times \overline{\mathcal{M}}_{0, l}\times \overline{\mathcal{M}}_{0, q}$ is a divisor for the birational reduction morphism $\rho _{\mathcal{A}_{0}, \mathcal{A}} : \overline{\mathcal{M}}_{0, n} \to \overline{\mathcal{M}}_{0, \mathcal{A}}$, where $\mathcal{A}_{0} = (1, \ldots ,1)$ and 
 $\mathcal{A}\in D_{\omega}\subset \mathcal{D}_{0, n}$, where  $C_{\omega}\subset \overline{D}_{\omega}$.

In addition   projection $p_{\omega}$ restricted to  $D_{s,l,q}$ coincides with  $\rho _{\mathcal{A}_{0}, \mathcal{A}} $ restricted  to $\mathcal{D}_{s,l,q}$.
 \end{prop}
\begin{proof}
If $\dim C_{\omega}=n-1$,  then   $\mathcal{O}(\mathcal{L}), \mathcal{L} = (t_1, \ldots , t_n)\in C_{\omega}$ is a typical linearisation.  So, there exists a unique chamber  $D_{\omega}$  in $\mathcal{D}_{0, n}$ such that $C_{\omega}\subset \overline{D}_{\omega}$.  Then $C_{\omega}$ and $D_{\omega}$ belong to the same open  chamber in $\R^{n}$ defined by the arrangement $\mathcal{W}_{c}$. If  $t_{i_1}+\ldots +t_{i_{s-1}}<1$ and $j_1+\ldots +j_{l-1}<1$  then  $a_{i_1}+\ldots +a_{i_{s-1}}<1$ and $a_{j_1}+\ldots +a_{j_{l-1}}<1$  so, according to Theorem~\ref{reduction} ,  $\overline{\mathcal{M}}_{0, s}\times \overline{\mathcal{M}}_{0, l}\times \overline{\mathcal{M}}_{0, q}$ is a divisor   for the reduction  $\rho_{\mathcal{A}_{0}, \mathcal{A}}$,  which maps by
\[
\overline{\mathcal{M}}_{0, s}\times \overline{\mathcal{M}}_{0, l}\times \overline{\mathcal{M}}_{0, q}\to \overline{\mathcal{M}}_{0, q}
\stackrel{\rho_{\mathcal{A}_{0}(q), \mathcal{B}}}{\lra}\overline{\mathcal{M}}_{0, \mathcal{B}},
\]
where $\mathcal{B}= (r_1, \ldots , r_{q-1}, 1,1)$ and $\{r_1, \ldots , r_{q-1}\}= \{1, \ldots, n\}\setminus \{i_1, \ldots , i_{s-1}, j_1, \ldots , j_{s-1}\}$.
According to Theorem~\ref{divisorsFn},  the last reduction coincides with the blowing down $\mathcal{F}_{q} \to \bar{F}_{q}$.

\end{proof}



Altogether we obtain:

\begin{thm}\label{projred}
 The projection map $p_{\omega} : \mathcal{F}_{n}\to F_{\omega}$ for $\dim C_{\omega}=n=-1$  coincides with the reduction morphism  $\rho_{\mathcal{B}, \mathcal{A}} : \overline{\mathcal{M}}_{0, \mathcal{A}} \to  \overline{\mathcal{M}}_{0, \mathcal{B}}$ for $\mathcal{A} = (1, \ldots, 1)$ and $\mathcal{B}\in D_{\omega}\subset \mathcal{D}_{0, n}$ such that $C_{\omega}\subset \overline{D}_{\omega}$. 
\end{thm}

We mention that for the forgetting morphism it is easy to verify::

\begin{prop}\label{forg}
The projection $r_{q} : \mathcal{F}_{n} \to \mathcal{F}_{n,q}$ $1\leq q\leq n$ and the forgetting morphism $\Phi _{\mathcal{A}_{0}, \mathcal{A}_{0}^{'}} : \overline{\mathcal{M}}_{\mathcal{A}_{0}} \to \overline{\mathcal{M}}_{\mathcal{A}_{0}^{'}}$ coincide,  where $\mathcal{A}_{0}=(1, \ldots , 1)$ and $\mathcal{A}_{0}^{'} =\pi ^{q}(\mathcal{A}_{0})$ is given by forgetting $q$-th coordinate.
\end{prop}



\section{Losev-Manin spaces and the model for  $G_{n,2}/T^n$}

As mentioned in the Introduction the Losev-Manin spaces $\overline{L}_{g, S}$ are the spaces of stable curves of genus $g$ endowed with a family of smooth painted by black or white  points  labeled by S satisfying that all white points are pairwise distinct and distinct from black ones, while black points may behave  arbitrarily. The number of white as well a as black points is assumed to be $\geq 2$.  In~\cite{M},  it is proved that $\overline{L}_{g, S}$ can be interpreted as the space $\overline{\mathcal{M}}_{g, \mathcal{A}}$ for appropriately chosen weights $\mathcal{A}$.
\begin{lem}
Let $S =W\cup  B$  be a painted set. Consider a weight data $\mathcal{A}$  on
$S$  satisfying the following conditions:
\begin{equation}\label{LMW}
 a_{s}=1\; \text{ for all}\;  s\in W,\; \text{and}\;  \sum _{t\in B} a_{t} \leq  1.
\end{equation}
Then a semistable $S$–pointed curve $(C, x_{s} | s\in  S)$  is painted stable if and only if
it is weighted stable with respect to $\mathcal{A}$.
\end{lem}

We consider the  spaces $\bar{L}_{0, S}$ and by previous Lemma obtain that $\bar{L}_{0, S} = \overline{\mathcal{M}}_{0, \mathcal{A}}$, where $\mathcal{A}$ satisfies~\eqref{LMW}.

For the clearness of the exposition we  first note the following, see also~\cite{H}.

\begin{lem}\label{aifree}
 The space $\overline{\mathcal{M}}_{0, \mathcal{A}}$, where  $\mathcal{A} = (1,\ldots ,1, a_{k+1},\ldots , a_n)$ and it holds
\begin{equation}\label{losevmaninweight}
0 <a_i < 1\; \text{for}\;  k+1\leq i\leq n \;  \text{and} \; \sum _{i=k+1}^{n} a_i\leq 1
\end{equation}
 does not depend on the choice of $a_i$'s which satisfy  the given conditions.
\end{lem}

\begin{proof}
Consider first the case when $a_{k+1}+\ldots +a_{n}<1$. Then $\mathcal{A}$ belongs to the open chamber for $\mathcal{D}_{0,n}$ defined by
\begin{equation}\label{chamber}
D_{0,n,k}: x_i + x_{l}>1\; \text{for}\; 1\leq i\leq k, 1\leq l\leq n, l\neq i, \;\; x_{k+1}+\ldots +x_{n}<1.
\end{equation}
 By Proposition~\ref{chambersweight}, the space 
 $\overline{\mathcal{M}}_{0, \mathcal{B}}$ does not depend on the choice $\mathcal{B}$ from this chamber. It follows that  $\overline{\mathcal{M}}_{0, \mathcal{A}}$ 
does not depend on the choice of $a_i$'s which satisfy~\eqref{losevmaninweight}.

Consider now the case when $a_{k+1}+\ldots +a_{n}=1$. Then $\mathcal{A}$ belongs to the chamber defined by $x_i+x_l>1$, $ 1\leq i\leq k, 1\leq l\leq n, l\neq i$ and   $ x_{k+1}+\ldots +x_{n}=1$. By Proposition~\ref{chambersweight}  the space $\overline{\mathcal{M}}_{0, \mathcal{A}}$ does not depend on choice of $a_i$'s such that $\mathcal{A}$ belongs to this chamber.

  We show that $\overline{\mathcal{M}}_{0, \mathcal{A}}$ is isomorphic to $\overline{\mathcal{M}}_{0, \mathcal{B}}$ for some 
weight vector $\mathcal{B}$ which belongs to a fine open chamber in $\mathcal{D}_{0,n}$. We follow the proof of Proposition 5.3 in~\cite{H}. Let $T_1$ , $T_{<1}$, $T_{>1}$ be the collection of all subsets $S\subset \{1,\ldots ,n\}$, $2\leq |S|\leq n-2$ such that $\sum _{i\in S}a_{i}=1$, $\sum _{i\in S}a_i<1$, $\sum _{i\in S}a_i>1$ respectively. More precisely,
\[
T_1 = \{S\}, \; \text{where}\; S=\{k+1, \ldots , n\},
\]
\[
T_{>1} = \{S | S=\{\underbrace{1,\ldots ,1}_{s}, a_{l_1},\ldots ,a_{l_r}\}, \; \text{where}\; 1\leq s\leq k, k+1\leq l_1,\ldots, l_r\leq n, 1\leq r\leq n-k\}
\]
\[
T_{<1}= \{ S | S=\{a_{l_1}, \ldots ,a_{l_r}\}, \; \text{where}\; k+1\leq l_1,\ldots, l_r\leq n, 2\leq s\leq n-k\}.
\]
 There exists $0<\varepsilon <1$ such that $\sum _{i\in S}a_i > 1+\varepsilon$ for any $S\in T_{>1}$, then $a_i>\varepsilon$ for $k+1\leq i \leq n$  and $k+a_{k+1}+\ldots a_k>2+\varepsilon$. Consider the weight vector
\[
\mathcal{B}= (1-\frac{\varepsilon}{n}, \ldots, 1-\frac{\varepsilon}{n}, a_{k+1}-\frac{\varepsilon}{n}, \ldots , a_{n}-\frac{\varepsilon}{n}).
\]
By Theorem~\ref{reduction} there exists birational reduction morphism $\rho _{\mathcal{B}, \mathcal{A}} : \overline{\mathcal{M}}_{0, \mathcal{A}}\to \overline{\mathcal{M}}_{0, \mathcal{B}}$.   This morphism is an isomorphism by making use of Corollary 4.7 from~\cite{H}.

We note that the vector $\mathcal{B}$ belongs to the chamber given by~\eqref{chamber}, which altogether proves the statement.  
\end{proof}

We denote by  $\overline{L}_{0, n, k}$ the Losev-Manin space for which $|W|=k$ and $|S|=n$, where $2\leq k\leq n-2$.
It follows that
\begin{cor}\label{LM-weight}
The space $\overline{L}_{0, n, k}$, $2\leq k<n$ coincides with the space $ \overline{\mathcal{M}}_{0, \mathcal{A}}$, where 
\begin{equation}\label{LML1}
\mathcal{A} = (\underbrace{1, \ldots , 1}_{k}, a_{k+1}, \ldots , a_n)
\end{equation}
for an arbitrary point $(a_{k+1},\ldots , a_{n})$ such that  $0 <a_i < 1\; \text{for}\;  k+1\leq i\leq n \;  \text{and} \; \sum _{i=k+1}^{n} a_i\leq 1$.
\end{cor}

\subsection{Losev-Manin spaces  $\overline{L}_{0, n, k}$ and spaces of parameters of chambers}

We want to show that the spaces   $\overline{L}_{0, n, k}$ can be, using  GIT theory, related by birational morphisms   to the spaces of parameters of some chambers for the Grassmannian $G_{n,2}$. We first emphasize that this can not be done by applying Theorem~\ref{omegaweight} or Theorem~\ref{omegabirat} to    the representation  $\overline{L}_{0, n, k}$ as the moduli space $\overline{\mathcal{M}}_{0, \mathcal{A}}$, where $\mathcal{A} = (1, \ldots, 1, a_{k+1}, \ldots , a_n)$. 

\begin{lem}
The closure of  the chamber $D_{0,n,k} \in \mathcal{D}_{0,n}$ defined by~\eqref{chamber} has empty intersection with any chamber $C_{\omega}\in \mathcal{C}_{n,2}$.  
\end{lem}

\begin{proof}
If $C_{\omega}\in \mathcal{C}_{n,2}$ then $C_{\omega}$ belongs to the hyperplane $x_1+\ldots +x_n=2$. It directly implies that $C_{\omega}\cap D_{0,n,k}=\emptyset$. A chamber $C_{\omega}$ can not intersect the boundary of $D_{0, n, k}$ since it would have to intersect a wall of $D_{0,n, k}$ which belongs to a hyperplane $x_i+x_j=1$, $1\leq i\leq k$, $k+1\leq j\leq n$, say the wall $x_1+x_{k+1}=1$. But for a point from this wall of $D_{0, n, k}$ it holds $x_2+x_{k+1}=1$ or $x_{2}+x_{k+2} >1$ and in both cases  it does not belong to $C_{\omega}$.
\end{proof}

Nevertheless,  we prove that $\overline{\mathcal{M}}_{0, \mathcal{A}}$ is birationally equivalent to a space of parameters $F_{\omega}$.
  Consider  the representation   of  $\overline{L}_{0, n, k}$ described by the moduli space whose weight vector is given by~\eqref{LML1} such that $\sum _{i=k+1}^{n}a_{i}<1$. 

 Let further  the weight vector $\mathcal{B}$ be defined by  
$\mathcal{B} = (1-\frac{\varepsilon}{n},\ldots ,1-\frac{\varepsilon}{n},a_{k+1},\ldots , a_{n})$, where $\varepsilon$ is  such that the  coordinates $b_i$, $1\leq i\leq n$ of the vector $\mathcal{B}$ satisfy
\begin{equation}\label{boundD}
0 < b_i <1\;  \text{for any}\;   1\leq i\leq k, \;\; 
 \sum _{i=1}^n b_i  =2,
\end{equation}
\begin{equation}\label{condtypical}
b_{i_1}+\ldots b_{i_l}  \neq 1
\end{equation}
for any  $2\leq i_1<\ldots i_l\leq n-1$.

Such $\varepsilon$ always exists. For example, assume that $(l-k)\sum _{s\in S} a_{s} + l\sum _{q\in Q} a_{q} \neq 2l-k$ for any $0\leq l\leq k$ and any division $\{k+1,\ldots n\} = P\cup Q$. The weights $a_i$'s can be chosen to satisfy these finite number of non-contradicting conditions.   Let $A_{0} = \sum _{i=k+1}^{n}a_{i}$ and  $a_0 = \frac{n(k-2+A_0)}{k}$ 	and put $\varepsilon = a_0$. It is easy to verify that the conditions~\eqref{boundD} are satisfied.


The  conditions~\eqref{boundD} and~\eqref{condtypical} immediately imply:
\begin{lem}
The weight  vector $\mathcal{B}$ belongs to $\partial \mathcal{D}_{0, n}$  and $\mathcal{B}$  is a typical linearisation.
\end{lem}

Then employing Theorem~\ref{typical} we deduce:

\begin{cor}
There exists $\delta >0$ such that 
\[
\overline{\mathcal{M}}_{0, \mathcal{B}_{\delta}} \cong \mathcal{G}(\mathcal{B})\; \text{where }\; \mathcal{B}_{\delta} = (b_1 +\delta, \ldots , b_{k}+\delta, a_{k+1},\ldots, a_{n})
\]
and $b_i+\delta < 1$ for $1\leq i\leq k$.
\end{cor}

\begin{rem}
The  weight $\mathcal{B}$ belongs to $\stackrel{\circ}{\Delta}_{n,2}$, that is to the interior of the admissible polytope $P_{k+1,\ldots, n}^{\leq 1}$ defined by the condition $x_{k+1}+\ldots +x_n\leq 1$. In addition, being typical,  $\mathcal{B}$  belongs  to  some chamber  $C_{\omega}$ in the interior of  $P_{1,\ldots, k}^{\leq 1}$ such that  $\dim C_{\omega} = n-1$. 

\end{rem}
Therefore, Theorem~\ref{omegagit}  and Theorem~\ref{omegaweight} give:
\begin{cor} 
There exists a chamber $C_{\omega}\subset \stackrel{\circ}{P}_{k+1, \ldots ,n}^{\leq 1}\subset \stackrel{\circ}{\Delta} _{n,2}$ of maximal dimension $n-1$ such that
\[
\overline{\mathcal{M}}_{0, \mathcal{B}_{\delta}}\cong F_{\omega},
\]
where $F_{\omega}$ is the space of parameters from $C_{\omega}$.
\end{cor}

Note also that any vector which belongs to a chamber of maximal dimension in $\stackrel{\circ}{P}_{k+1, \ldots, n}^{\leq 1}$    satisfies conditions~\eqref{boundD} and~\eqref{condtypical}.  Then the above can be applied  to any such vector due to Lemma~\ref{aifree}.

Altogether, reduction Theorem~\ref{reduction} finally  implies:

\begin{thm}\label{LMChn-1}
For a Losev-Manin space $\overline{L}_{0, n, k}$ and any  chamber $C_{\omega}$ which belongs to $\stackrel{\circ}{P}_{k+1,\ldots ,n}^{\leq 1}$ and  $\dim C_{\omega}=n-1$,  there exists   birational morphism 
\[
\overline{L}_{0, n, k} \to F_{\omega},
\]
given by the birational reduction morphism $ \overline{\mathcal{M}}_{0, \mathcal{A}} \to \overline{\mathcal{M}}_{0, \mathcal{B}_{\delta}}$, where
$\mathcal{A} = (\underbrace{1, \ldots , 1}_{k}, a_{k+1}, \ldots , a_{n})$ and  $\mathcal{B}_{\delta} = (b_1 +\delta, \ldots , b_{k}+\delta, a_{k+1},\ldots, a_{n})$ for $\mathcal{B} = (b_1,\ldots , b_{k}, a_{k+1},\ldots , a_{n})\in C_{\omega}$.
\end{thm}

Let us now consider a vector  $\mathcal{B} = (b_1, \ldots , b_n)$ which satisfies~\eqref{boundD}, but which does not satisfy~\eqref{condtypical}. Then  $\mathcal{B}$ is an  atypical linearisation.  Let $b_{s}+\ldots +b_n<1$ for some $3\leq s\leq n-2$.  Then for each coarse chamber in $\mathcal{D}_{0,n}$ which contains $\mathcal{B}$ in its closure, we can find $\delta _{1}, \ldots , \delta _{s-1}$ such that  $\tilde{B}_{\delta} = (b_1+\delta _1, \ldots ,
b_{s-1} +\delta _{s-1}, b_1, \ldots , b_n)$ belongs to that chamber. Now, consider $\mathcal{A} = (1,\ldots ,1, b_s, \ldots , b_n)$ and the Losev-Manin space $\overline{L}_{0, n, k}$. 
Then Theorem~\ref{atypical} and Theorem~\ref{reduction} imply that there exists a birational morphism $\overline{L}_{0, n, k} \to \overline{\mathcal{M}}(0, \mathcal{B}_{\delta})\to \mathcal{G}(\mathcal{B})$ given as the composition of birational morphisms. Altogether we deduce:

\begin{thm}\label{LMChn-2}								
For any chamber $C_{\omega}$ in $\stackrel{\circ}{\Delta}_{n,2}$ such that  $\dim C_{\omega}\leq n-2$ there exists a Losev-Manin space $\overline{L}_{0, n, k}$ such that there exists birational morphism
\[
\overline{L}_{0, n, k} \to F_{\omega}.
\]
\end{thm}

\subsection{Losev-Manin spaces and wonderful compactification}

Let $\mathcal{A}_{0} = (1, \ldots , 1)$ and $\mathcal{A} = (1,\ldots , 1, a_{k+1}, \ldots , a_n)$. The natural birational reduction morphism $\rho _{\mathcal{A}_{0}, \mathcal{A}}: \overline{\mathcal{M}} _{0, \mathcal{A}_{0}}\to \overline{\mathcal{M}}_{0, \mathcal{A}}$ gives the birational  morphism $\rho _{n, k} : \mathcal{F}_{n}\to L_{0, n, k}$. By Theorem the divisors of maximal dimension for this morphism are given by $D_{I, J}  \cong \mathcal{F}_{r+1}\times \mathcal{F}_{n-r+1}$, where $I \subset \{k+1,\ldots ,n\}$ and $|I|=r$ for $3\leq r\leq n-k$,  while $J=\{1, \ldots, n\}\setminus I$. The morphism $\rho _{n, k}$ factorizes on $D_{I, J}$ as
\begin{equation}\label{divred}
\mathcal{F}_{r+1}\times \mathcal{F}_{n-r+1} \stackrel{\pi}{\to} \mathcal{F}_{n-r+1} \stackrel{\rho _{n-r+1, k}}{\to}   L_{0, n-r+1, k}.
\end{equation}
For $r=n-k$ we obtain the following contraction:
\[
\mathcal{F}_{n-k+1}\times \mathcal{F}_{k+1} \stackrel{\pi}{\to} \mathcal{F}_{k+1}\stackrel{\rho _{k+1, k}}{\to}   L_{0, k+1, k}, \]
where the second morphism is an isomorphism.

It follows from~\eqref{divred} that,  for a fixed $k$,  the morphism $\rho _{n,k}$   can be described inductively on $n$.

We consider now the case  $k=2$.  For low $n$ we have:

\begin{lem}\label{lem:LM}
The Losev -Manin space $L_{0, n,2}$ is  
\begin{enumerate}
\item for $n=5$  isomorphic to $\mathcal{F}_{5}/\mathcal{F}_{4}$;
\item for $n=6$  obtained from $\mathcal{F}_{6}$ by contracting  the divisor $\mathcal{F}_{5}$ to a point and four divisors isomorphic to $\mathcal{F}_{4}\times \mathcal{F}_{4}$ to $\mathcal{F}_{4}$;
\item for $n=7$ obtained from $\mathcal{F}_{7}$ by contracting the divisor $\mathcal{F}_{6}$ to a point, each of ten  divisors $\mathcal{F}_{5}\times \mathcal{F}_{4}$ to $\mathcal{F}_{4}$ and each of five divisors $\mathcal{F}_{4}\times \mathcal{F}_{5}$ to $\mathcal{F}_{5}/\mathcal{F}_{4}$.
\end{enumerate}
\end{lem}
\begin{proof}
The weight vector which corresponds to $L_{0, 5, 2}$ is $\mathcal{A}=(1,1,a_3,a_4,a_5)$, where $a_3+a4+a_5 <1$. The reduction birational morphism has one divisor $D= \mathcal{F}_{4}\times \mathcal{F}_{3}$ which contracts to $\mathcal{F}_{3}$,  that is to a point

The reduction birational morphisms $\rho _{6,2} : \mathcal{F}_{6} \to L_{0, 6, 2}$ has five divisors given as follows: for $r=4$ it is $\mathcal{F}_{5}\times \mathcal{F}_{3}$ and for $r=3$,  there are four divisors corresponding to the three element subsets of $\{3,4,5,6\}$ each isomorphic to $\mathcal{F}_{4}\times \mathcal{F}_{4}$. The reduction morphism contract $\mathcal{F}_{5}$ to a point and each of $\mathcal{F}_{4}\times \mathcal{F}_{4}$ to $\mathcal{F}_{4}$.

For $n=7$ the divisors  are as follows: for $r=5$,   it is $\mathcal{F}_{6}$ and it contracts to $\mathcal{F}_{3}$, that is to a point;  for $r=4$,  there are ten divisors $\mathcal{F}_{5}\times \mathcal{F}_{4}$ and they contract to $\mathcal{F}_{4}\stackrel{\rho _{4,2}}{\to}L_{0, 4,2}$ that is to $\mathcal{F}_{4}$; for $r=3$,  there are five divisors $\mathcal{F}_{4}\times \mathcal{F}_{5}$ and they contract to $\mathcal{F}_{5}\stackrel{\rho _{5, 3}}{\to}L_{0, 5, 3}$ that  is to $\mathcal{F}_{5}/\mathcal{F}_{4}$.
\end{proof}

Recall~\cite{BT2} once again   that  the universal space of parameters $\mathcal{F}_{n}$ can be obtained as the wonderful compactification of $\bar{F}_{n}$ with the building set 
$\mathcal{G}_{n}$ which consists of all  possible nonempty intersections of the  $\hat{F}_{I} =\bar{F}_{n}\cap \{(c_{ij}: c_{ij}^{'}) =  (c_{ik}:c_{ik}^{'})  =(c_{jk}:c_{jk}^{'})=(1:1)\}$, where $I = \{i,j,k\} \in \{ I\subset \{1, \ldots , n\}, |I|=3\}$ and $n\geq 5$. This compactification is obtained as iterated blow ups along the varieties from  this building set.  We  list the divisors in this compactification for small $n$.

\begin{lem}
The divisors in the wonderful compactification of $\bar{F}_{n}$ to $\mathcal{F}_{n}$ are:
\begin{enumerate}
\item $\C P^1$  for $n=5$;
\item $\mathcal{F}_{5}$ and four copies of $\mathcal{F}_{4}\times \mathcal{F}_{4}$ for $n=6$;
\item $\mathcal{F}_{6}$, ten copies of $\mathcal{F}_{5}\times \mathcal{F}_{4}$ and five copies $\mathcal{F}_{4}\times \mathcal{F}_{5}$ for $n=7$.
\end{enumerate}
\end{lem}

\begin{proof}
The universal space of parameters $\mathcal{F}_{5}$ is the blow-up of $\bar{F}_{5}$ at the point $((1:1), (1:1), (1:1))$ so the divisor is $\C P^{1}$.
The space $\mathcal{F}_{6}$ is the  iterated blow-up of $\bar{F}_{6}$  along  the subvarieties $\hat{F}_{345}, \hat{F}_{346}, \hat{F}_{356}, \hat{F}_{456}$, which are isomorphic to $\C P^{1} \cong \mathcal{F}_{4}$  and the point $S = (1:1)^{6}$. The point $S$ is the intersection point of all these subvarieties and each two of them as well. Therefore, there are four divisors isomorphic to $\mathcal{F}_{4}\times \mathcal{F}_{4}$.

For $n=7$ the subvarieties $\hat{F}_{I} \subset \bar{F}_{7}$ are $\hat{F}_{345}, \hat{F}_{346}, \hat{F}_{347}, \hat{F}_{356}, \hat{F}_{357}, \hat{F}_{367}, \hat{F}_{456}, \hat{F}_{457}$, $\hat{F}_{567}$ and each of them is isomorphic to $\bar{F}_{5}$. The intersection of any of these  subvarieties is  the point $S=(1:1)^ {10}$ or it is  isomorphic to $\bar{F}_{4}$. More precisely, $\hat{F}_{I}\cap \hat{F}_{J}\cong \bar{F}_{4}$ for $|I \cap J|=2$, while $\hat{F}_{I}\cap \hat{F}_{J}=S$ for $|I\cap J|=1$.  There are ten choices for the first case and  the five choices for the second case.  As the result of the iterated blow up procedure we obtain the following divisors: $\mathcal{F}_{6}$ at the point $S$, ten copies of $\mathcal{F}_{5}\times \mathcal{F}_{4}$ along  subvarieties $\hat{F}_{I}\cap \hat{F}_{J} \cong \bar{F}_{4}$ and five copies of $\mathcal{F}_{4}\times \mathcal{F}_{5}$ 
along subvarieties $\hat{F}_{I}$. 
\end{proof}

In an analogous way it can be proved the generalization for an arbitrary $n$:

\begin{thm}\label{divisorsFn}
The divisors for the reduction morphism $\rho _{n, 2} : \mathcal{F}_{n} \to L_{0, n, 2}$ coincide with the divisors in the wonderful compactification   of $\bar{F}_{n}$ which produces $\mathcal{F}_{n}$. 
Thus, the  Losev-Manin $L_{0, n,2}$ is isomorphic to $\bar{F}_{n}\subset (\C P^{1})^{N}$, $N = \binom{n-2}{2}$.
\end{thm}

\begin{cor}\label{barFn}
 The space $\bar{F}_{n}$ is isomorphic to $\overline{\mathcal{M}}_{0, \mathcal{A}}$ for $\mathcal{A} = (1, 1, a_3, \ldots , a_n)$, where $\sum _{i=1}^{n}a_i \leq 1$.
\end{cor}

We show that the  manifold   $\bar{F}_{5}$ is a subspace of $G_{n,2}/T^n$, that is $L_{0,5,2}$ can be realized as a  subspace in $G_{5,2}/T^5$.

\begin{prop}
The space $\bar{F}_{5}$  is isomorphic to $F_{\omega}$ for the chamber $C_{\omega} =\{ (t_1,\ldots , t_5)\in \Delta _{5,2} | t_1+t_2 >1, \; t_i+t_j<1, \; \{i,j\}\neq \{1,2\}\}$. 
\end{prop}
\begin{proof}
Let $\mathcal{A} = (1,1,\frac{5}{18}, \frac{5}{18},\frac{5}{18})$, then $L_{0,5,2} \cong \mathcal{M}_{0, \mathcal{A}}$. Note that $\mathcal{A}$ belongs to the chamber in $\R ^{5}$ given by $x_1+x_i>1, x_2+x_i>1, x_3+x_4+x_5<1$.  Let $\mathcal{L}=(\frac{3}{6}, \frac{4}{6}, \frac{5}{18}, \frac{5}{18}, \frac{5}{18})$, then $\mathcal{L}\in \stackrel{\circ}{\Delta}_{n,2}$ and $\mathcal{L}$ is a typical linearisation. It belongs to the chamber $C_{\omega}$ in $\stackrel{\circ}{\Delta}_{n,2}$ given by $t_1+t_2>1$ and $t_i+t_j<1$ for the others. It follows that 
\[
F_{\omega} \cong \mathcal{G}(\mathcal{L}) \cong \mathcal{M}_{0, \mathcal{B}}\;\; \text{for}\;\; \mathcal{B}\in \mathcal{U}(\mathcal{L})\cap \mathcal{D}_{0,5}.
\]
On the other hand, the space of parameters $F_{\omega}$ is $\bar{F}_{5}$. We can see it easily as follows: the admissible polytopes which form the chamber $C_{\omega}$ are the hypersimplex $\Delta _{n,2}$ whose space of parameters is $F_5$  and the polytopes   $t_i+t_j<1$ for $\{i, j\}\neq \{1,2\}$ whose spaces of parameters are $F_4$,  then the polytope $t_4+t_5+t_6< 1$ and the polytopes given by the intersection of $t_i +t_j<1, t_k+t_l<1$, $\{i,j\}, \{k, l\}\neq \{1,2\}$ and $\{i,j\}\cap \{k,l\} =\emptyset$ whose spaces of parameters are points.  It easy to verify that these nine copies of $F_4$ together with these $14$ points compactify $F_{5}$ to $\bar{F}_{5}$.

Since $\bar{F}_{5}$ is obtained form $\mathcal{F}_{5}$ by blowing down $\C P^1 = \mathcal{F}_{4}$ it follows from Lemma~\ref{lem:LM} that $L_{0,5,2}$ and $\bar{F}_{5}$ are isomorphic.
\end{proof}

Since $F_{\omega}\cong \hat{\mu}^{-1}(t)$ for $t\in C_{\omega}$ we deduce:

\begin{cor}\label{F5}
The Losev-Manin space $L_{0,5,2}$ is contained the orbit  space $G_{5,2}/T^5$, that is
\[
L_{0,5,2}\cong (G_{5,2}\setminus (W_{12}\cup W_{12, 34}\cup W_{12, 35}\cup W_{12. 45}))/(\C ^{\ast})^5.
\]
\end{cor}
\begin{proof}
By Theorem~\ref{dimensionn-1} the chamber $C_{\omega} =\{ (t_1,\ldots , t_5)\in \Delta _{5,2} | t_1+t_2 >1, \; t_i+t_j<1, \; \{i,j\}\neq \{1,2\}\}$ is obtained as the intersection of all admissible polytopes apart from $P_{12}, P_{12, 34}, P_{12, 35}, P_{12, 45}$ which proves  the statement.
\end{proof}

\begin{rem}
The previous statement does not hold for $n\geq 5$, that is the space $\bar{F}_{n}$ for $n\geq 6$ can not be obtained as the space of parameters  $F_{\omega}$ for some chamber $C_{\omega}\subset \Delta_{n,2}$. In more detail, on the one hand by Corollary~\ref{barFn} we see that $\bar{F}_{n}$ can be obtained  from the reduction morphism  $\rho _{n,k} : \mathcal{F}_{n}\to \bar{L}_{0,n,2}$ by collapsing the divisors $\mathcal{F}_{r+1}\times \mathcal{F}_{n-r+1}$ according to formula~\eqref{divred}, where the factors $\mathcal{F}_{r+1}$ are determined by all  $\{a_{i_1}, \ldots , a_{i_r}\}\subset \{a_{3}, \ldots , a_{n}\}$, $r\geq 3$. On the other hand   if $C_{\omega}\subset P_{\sigma}$, where $P_{\sigma}$ is an admissible polytope given  by $t_{3}+\ldots + t_{n}\leq 1$, then there is  no  neighborhood $U$  of $(t_1, \ldots , t_n)\in \stackrel{\circ}{P}_{\sigma}$ such that for $(x_1, \ldots , x_n)\in U$ it holds $x_1+x_3+\ldots + x_n >1$ and $x_2+ x_3+\ldots + x_n > 1$. Namely, for some $n_0$ we would have that $(t_1+\frac{1}{n}, t_2, \ldots , t_n)\in U$ and,  since $t_1+\ldots +t_n =2$ for  $t_2+ t_3+ \ldots +t_n >1$ we have $t_1+t_3+\ldots + t_n<1$ which implies $t_1+\frac{1}{n_1}+t_1+\ldots +t_n<1$ for large enough $n_1$. Therefore, for any neighborhood $U$ of $(t_1, \ldots ,t_n)$  the reduction morphism $\mathcal{F} \to \overline{\mathcal{M}}_{0, (a_1, \ldots, a_n)}$, $ (a_1, \ldots ,a_n)\in U$ can not coincide with $\rho _{n,2}$. 

If $C_{\omega}\subset P_{\sigma}$ for all  $P_{\sigma}$  determined by $x_{i_1}+\ldots + x_{i_{n-2}}>1$, being equivalent that  $C_{\omega}\subset P_{\sigma}$  for all  $x_i+x_j\leq 1$, then $F_{\omega}$ contains $\binom{n}{2}$ exemplars  of $F_{n-1}$, since the space of parameters of the stratum whose admissible polytope if $x_i+x_j<1$ is $F_{n-1}$. On the other side,   it easily verifies that $\bar{F}_{n}$ contains $\binom{n-2}{1}-1$ such exemplars.  
\end{rem}

\subsection{Losev-Manin spaces and toric varieties}

It  is the result of~\cite{LM} that   the Losev-Manin spaces $L_{0, n, 2}$ are toric varieties. In addition, following~\cite{KAP} the toric varieties $L_{0, n,2}$ are  the permutohedral varieties $\chi (P_{e}^{n-3})$, that is the toric varieties over a  permutohedron $P_{e}^{n-3}$.  Thus,  by~\cite{GS} a variety  $L_{0, n, 2}$ can be obtained as  the closure of a principal $(\C ^{\ast})^{n-3}$-orbit in a complex complete flag manifold $Fl(n-3)$.

From  Theorem~\ref{LMChn-1} and Corollary~\ref{F5} we deduce:

\begin{cor}

\begin{itemize}

\item The permutohedral variety $\chi (P_{e}^{2})$ can be isomorphically mapped to the  variety in the orbit  space $G_{5,2}/T^5$.
\item A permutohedral variety $\chi (P_{e}^{n-3})$ can be mapped by birational morphism to a variety in the orbit space $G_{n,2}/T^n$ for $n\geq 6$.
\end{itemize}
\end{cor}

\begin{rem}
To any graph $\Gamma$ can be assigned a toric variety $\chi (\mathcal{P}\Gamma)$. In more detail for a graph on $n+1$ vertices the graph associahedron $\mathcal{P}\Gamma$ is a $n$-dimensional convex polytope constructing by truncating the $n$-simplex based on the connected subgraphs of $\Gamma$. These polytopes include  as well permutohedron, cyclohedron  associahedron and stellahedron. The toric variety  $\chi (\mathcal{P}\Gamma)$ is obtained by blowing up a projective space along coordinate subspaces which correspond to the connected subgraphs of $\Gamma$ and $\mathcal{P}\Gamma$ is the polytope of this toric variety. In the paper~\cite{FJR} it is obtained complete classification of graphs $\Gamma$ such that the toric variety $\chi (\mathcal{P}\Gamma)$ is isomorphic to a Hassett space $\overline{\mathcal{M}}_{0, \mathcal{A}}$. More precisely, it is proved that for a graph $\Gamma$ on $n-2$ vertices, the toric variety $\chi (\mathcal{P}\Gamma)$ is isomorphic to a Hassett space $\overline{\mathcal{M}}_{0, \mathcal{A}}$ if and only if $\Gamma$ is an iterated cone over a discrete set, meaning that $\Gamma = \text{Cone}^{n-k-2}(\cup _{i=1}^{k}v_i)$. 
\end{rem}

The classification theorem from~\cite{FJR} implies that there are not so much intersection between the theory of Hassett spaces and the theory of  toric varieties  over graph associahedra. For example,  an  associahedron or  cyclohedron   can not be obtained as iterated cones over discrete sets. 

\begin{rem}
The  complete graph $K_{n-2}$ is obtained as  $(n-3)$-times iterated cone over a single vertex and it produces  the permutohedral variety, that is the Losev-Manin space $\bar{L}_{0,n,2}$.  The star graph  is the cone over the discrete set on $n-2$ vertices and  its  graph associahedron is stellahedron.  In~\cite{FJR}, Corollary 2,  it is shown that the corresponding   toric graph associahedra is  isomorphic to the  moduli space   $\overline{\mathcal{M}}_{0, \mathcal{A}}$, where $\mathcal{A} =(1, \frac{1}{2}, \frac{1}{2}+\varepsilon, \varepsilon, \ldots , \varepsilon) \in \R^{n}$ with $\varepsilon <\frac{1}{n}$. 
\end{rem}

Combining this with our result we deduce:
\begin{cor}
 The graph associahedron toric variety $\chi (\mathcal{P}\Gamma)$ which corresponds to a graph $\Gamma  = \text{Cone}^{n-k-2}(\cup _{i=1}^{k}v_i)$ can be mapped by birational morphism to a variety in the orbit space $G_{n, 2}/T^n$.
\end{cor}

\subsubsection{Losev-Manin spaces and permutohedral varieties}

We need the following result in order to provide geometric description of permutohedral toric varieties.

\begin{lem}
It is defined the representation  of $(\C ^{\ast})^{n-3}$  in $(\C ^{\ast})^{N}$,  $N = \binom{n-2}{2}$ by
\[
\rho (t_1, \ldots , t_{n-3}) = ( \rho_{ij}(t_1, \ldots ,t_{n-3})), \; 3\leq i <j\leq n,
\]
where $\rho _{3k}(t_1, \ldots , t_{n-3}) = t_{k-3}$ for $4\leq k \leq n$ and 
\[
\rho _{ij}(t_1, \ldots , t_{n-3}) = \frac{t_{j-3}}{t_{i-3}}, \;\; 4\leq i<j\leq n-3.
\]
\end{lem}

\begin{proof} 
We see  that the map $\rho$ is $1-1$ and  
\[
\rho _{ij}( (t_1, \ldots , t_{n-3})\cdot (t_1^{'}, \ldots , t_{n-3}^{'})) = \rho _{ij} (t_1t_1^{'}, \ldots , t_{n-3}t_{n-3}^{'}) =\]
\[ \frac{t_{j-3}t_{j-3}^{'}}{t_{i-3}t_{i-3}^{'}} = \rho _{ij}( (t_1, \ldots , t_{n-3})\cdot \rho _{ij}( (t_1^{'}, \ldots , t_{n-3}^{'}).
\]
\end{proof}

Consider the $(\C^{\ast})^{n-3}$-action on $(\C P^{1})^{N}$ defined by the composition of the canonical action of $\C ^{\ast}$ on $\C P^1$  and the diagonal action of   $(\C^{\ast})^{n-3}$ on  $(\C P^{1})^{N}$  given by the  representation   $\rho ((\C ^{\ast})^{n-3})$, that is 
for ${\bf t} = (t_1, \ldots, t_{n-3})$ we have 
\begin{equation}\label{action}
{\bf t}\cdot ((c_{34}:c_{34}^{'}), \ldots ,(c_{n-1n}:c_{n-1n}^{'})) = ((\rho _{34}({\bf t})c_{34} : c_{34}^{'}), \ldots , (\rho _{n-1n}({\bf t})c_{n-1n}: c_{n-1n}^{'})).
\end{equation}

\begin{lem}
The compactification $\bar{F}_{n}$  of $F_n=\mathcal{M}_{0,n}$ in $(\C P^{1})^{N}$ given by~\eqref{barFn} is invariant under the $(\C ^{\ast})^{n-3}$-action given  by~\eqref{action}.
\end{lem}

\begin{proof}
The points $((c_{ij}: c_{ij}^{'})_{3\leq i<j\leq n})\in \bar{F}_{n}$ are those points from $(\C P^{1})^{N}$ which satisfy equations $c_{ij}c_{ik}^{'}c_{jk} = c_{ij}^{'}c_{ik}c_{jk}^{'}$. Now, the  $(\C ^{\ast})^{n-3}$-orbit of a point from $\bar{F}_{n}$ is $((\rho _{ij}({\bf t})c_{ij}
:c_{ij}^{'})$, so  for $i=3$ 
\[
\rho _{3j}({\bf t})c_{3j}c_{3k}^{'} \rho _{jk}({\bf t})c_{jk} = t_{k-3}c_{3j}c_{3k}^{'}c_{jk}=
 c_{3j}^{'}\rho_{3k}({\bf t})c_{3k} c_{jk}^{'}.
\]
 For $i>3$ we have  
\[
\rho _{ij}({\bf t})c_{ij}c_{ik}^{'} \rho _{jk}({\bf t})c_{jk} = \frac{t_{k-3}}{t_{i-3}}c_{ij}c_{ik}^{'}c_{jk} = c_{ij}^{'}\rho _{ik}({\bf t})c_{ik} c_{jk}^{'}.
\]
\end{proof}
According to the usual convention we say that a ${\bf c}\in \bar{F}_{n}$ is a general point   for the $(\C ^{\ast})^{n-3}$ action if its stability is trivial  and for a such orbit we say to be principal.

\begin{prop}\label{Fnorbit}
 The space $\bar{F}_{n}$ is a smooth toric variety for $(\C ^{\ast})^{n-3}$-action 
 and its principal orbit contains the space  $F_{n}\cong \mathcal{M}_{0, n}$.
\end{prop}

\begin{proof}
The $(\C ^{\ast})^{n-3}$-orbit of a point ${\bf c}\in F_{n}$ is principal, since its stabilizer is obviously trivial. Moreover, it is straightforward to see that  $F_{n}\subset (\C ^{\ast})^{n-3}\cdot {\bf c}$. Since the closure of $F_{n}$ in $(\C P^{1})^{N}$ is $\bar{F}_{n}$,  it follows  that
$\overline{(\C ^{\ast})^{n-3}\cdot {\bf c}}$ in $(\C P^{1})^{N}$ is $\bar{F}_{n}$.
\end{proof}

Theorem~\ref{divisorsFn} states that the  Losev-Manin compactification of $F_n$ gives the smooth manifold $\bar{F}_{n}$.   From~\eqref{jednacineFn} and~\eqref{jednacinebarFn} it is straightforward to deduce:

\begin{lem}
The outgrows of of the Losev-Manin compactification $L_{0, n, 2}$  for  $\mathcal{M}_{0, n}$  can be stratified as $\mathcal{M}_{0, n_1}\times \ldots \times  \mathcal{M}_{0, n_k}$, where $n_1+\ldots +n_k= n+2k-2$.
\end{lem}

Since $F_n$ is contained in the principal orbit for $(\C ^{\ast})^{n-3}$-action on $\bar{F}_{n}$ it naturally arises the question on  a relation between the outgrows of Losev-Manin compactification of  $F_n$ to $\bar{F_n}$ and and outgrows  of the toric compactification of $(\C ^{\ast})^{n-3}\cdot {\bf c}$ to $\bar{F}_{n}$ for $c\in F_n$

\begin{prop}\label{outgrows}
The outgrows of the Losev-Manin compactification  can be divided into two sets such that:
\begin{itemize}  
\item those from  the first set give the completion that is extension  of $F_n$ to $(\C ^{\ast})^{n-3}\cdot {\bf c}$ for ${\bf c}\in F_n$ and they contain the points having at least one coordinate $(1:1)$ and do not contain  the coordinate $(1:0)$ or $(0:1)$,
\item  those from the second  set give the toric compactification of $(\C ^{\ast})^{n-3}\cdot {\bf c}$ to $\bar{F}_{n}$ and they contain the points having at least one coordinate $(1:0)$ or $(0:1)$.
\end{itemize}
\end{prop}

\begin{proof}
The outgrows in the  toric compactification of $(\C ^{\ast})^{n-3}\cdot {\bf c}$ to $\bar{F}_{n}$ are given by the varieties determined by the condition that at   least one coordinate   $(c_{ij}:c_{ij}^{'})$, and consequently, due to the equation~\eqref{jednacineFn},   more of them, is commonly equal to  $ (1:0)$ or $(0:1)$. The outgrows of the Losev-Manin  compactification of $F_n$ to $\bar{F}_{n}$   are given by the varieties  determined by the condition that the   least one coordinate   $(c_{ij}:c_{ij}^{'})$, and consequently more of  them, is commonly equal to  $ (1:0)$, $(0:1)$ or $(1:1)$. Obviously, the completion of $F_{n}$ by the varieties determined by the condition that coordinates in the defining  system~\eqref{jednacineFn} for $F_n$  all allowed to be form $\C P^{1}\setminus \{(1:0), (0:1)\}$ give the principal orbit, since the point from $\bar{F}_{n}$ whose all coordinates are $(1:1)$ is a general point.
\end{proof} 

The division of the outgrow points given by Proposition~\ref{outgrows} can be reformulated as follows:

\begin{cor}
The  points which belong to the outgrows of the Losev-Manin compactification can be divided into two classes such that one  class consists of the points whose orbits are principal and the other class  consists of the  remaining points. 
\end{cor}

Using equations~\eqref{jednacinebarFn}, one gets the explicit  description of the outgrows from the both the first and  second set in the Losev-Manin compactification stated by Proposition~\ref{outgrows}.   Then the direct computation of their numbers leads to the proof of the following important result.

\begin{thm}
The smooth toric varieties $\bar{F}_{n}\subset (\C P^{1})^{N}$ are the  permutohedral toric varieties $\chi (P_{e}^{n-2})$.
\end{thm}
 
Our description of the space $\mathcal{F}_{n}$, which we proved to be diffeomorphic to $\overline{\mathcal{M}}_{0, n}$, as the wonderful compactification of $\bar{F}_{n}$ can be now reformulated as:

\begin{cor}\label{DMLM}
The Deligne-Mumford compactification $\overline{\mathcal{M}}_{0, n}$ can be obtained by the wonderful compactification of the Losev-Manin spaces $\bar{L}_{0,n, 2}$.
\end{cor}

\begin{rem}
The outgrows of the  Deligne-Mumford compactification  for  $F_n$ are also  given by   $\mathcal{M}_{0, n_1}\times \ldots \times  \mathcal{M}_{0, n_k}$ but they do not coincide with those given by Losev-Manin compactification neither in number neither in form.  Using Corollary~\ref{DMLM},  we can observe that those outgrows in the Deligne-Mumford compactification which map $1-1$ by the natural blow-down projection are the outgrows in the Losev-Manin compactification as well. We demonstrate it for $n=5$ in Example~\ref{ex5}.
\end{rem}

\subsection{Inversion formula and Losev-Manin  spaces}

\subsubsection{Differential algebra of smooth toric varieties}
In~\cite{BDGP},  it is introduced the differential algebra  $\mathcal{P}$ of simple polytopes graded by the dimension of the polytopes. The differential of a simple polytope $P$ is given as the sum, that is disjoint union,   of all facets of $P$, and the Liebnitz  formula holds $d(P_1P_2) = d(P_1)P_2+P_1d(P_2)$.

We introduce here the differential  algebra   $\mathcal{T}$ of smooth toric varieties. Denote by $\mathcal{T}_{2n}$ abelian group which is generated by   equivariantly diffeomorphic  smooth toric varieties and the sum is given by the  formal union of these varieties, while the identity is an empty set. By $\mathcal{T} =\sum _{n\geq 1}\mathcal{T}_{2n}$  we denote the differential graded associative and commutative algebra, where the product is given by the direct product of varieties and the graduation is given by the dimension.  

The differential is given as follows. For any smooth toric variety  $M^{2n}$ there is a moment  map $\mu_{M^{2n}} : M^{2n} \to P^{n}$   to a simple polytope $P^{n}$. For any facet $P^{n-1}$ of $P^{n}$
the preimage $\mu _{M^{2n}}^{-1}(P^{n-1})$ is  a smooth toric variety.  Therefore, we define the   differential of a smooth toric variety $M^{2n}$  to be the sum of all $\mu _{M^{2n}} ^{-1}(P^{n-1})$ for all facets $P^{n-1}$ of $P^{n}$.  The following is obvious.

\begin{lem}\label{homomor}
There exists  the natural  differential algebra homomorphism $\mathcal{\mu} : \mathcal{P}\to \mathcal{T}$ given  by
\[
\mathcal{\mu}(M^{2n})  = \mu _{M^{2n}}(M^{2n}).
\]
\end{lem}

This  homomorphism is not surjective.  We emphasize that  there are subalgebras  in the algebra $\mathcal{P}$, consisting of graph associahedra or    truncated cubes,  which belong to the image of this homomorphism.

In analogy to an algebra $\mathcal{P}$,  it can be defined the generating function $F_{M^{2n}}(t)\in \mathcal{T}[t]$ for a smooth toric manifold $M^{2n}$ by
\[
F_{M^{2n}}(t) = \sum _{M^{2k} \subset \partial M^{2n}}M^{2k} t^{n-k},
\]
where the sum goes over all smooth toric varieties $M^{2k} = \mu ^{-1}_{M^{2n}}(P^k)$ for a face $P^k$ of $P^n$.  According to~\cite{BV}, the value $F_{M^{2n}}(-1)\in \mathcal{T}$ we call  the Euler interior of $M^{2n}$ and denote by $\stackrel{\circ}{M}^{2n}$.

\subsubsection{Inversion formulas}
It is the well known classical question:   for a given power series $f(x) = x+ \sum _{n\geq 2}a_nx^n$  find its compositionally inverse $g(y)$, that is    $f(g(y)) = y$ and $g(f(x)) = x$ or,   for a given power series  $f(x) = 1+ \sum _{n\geq 1} a_nx^n$ find its multiplicative inverse $g(x)$, that is $f(x)\cdot g(x)=1$.  

 It is that McMullen~\cite{McM}  proved that the compositionally inverse of the exponential power series can be linked to the geometry  of the moduli space $\overline{\mathcal{M}}_{0,  n}$. More precisely, he proved that the compositionally inverse of the formal series $f(x) = x-\sum _{n\geq 2}a_n\frac{x^{n}}{n!}$ is given by 
\[
g(x) = x+\sum _{n\geq 2}b_n\frac{x^{n}}{n!}, \;\; b_{n} = \sum N_{n_1\ldots n_s}a_{n_1}\cdots a_{n_s},
\]
where $N_{n_1\ldots n_s}$ is the number of strata $S\subset \overline{\mathcal{M}}_{0, n+1}$ isomorphic to $\mathcal{M}_{0, n_1+1}\times \cdots \times \mathcal{M}_{0, n_s+1}$ and $n_1+\ldots +n_s = n+s-1$.  

\begin{rem}
In the paper~\cite{BV} it is proved by Corollary 3.4 that the coefficients $b_n$ of the compositionally  inverse $g(x) = x+\sum _{n\geq 2}b_nx^n$ to a formal power series $f(x) = x +\sum _{n\geq 2}a_nx_n$ can be described  in terms  of the coefficients $a_n$ and the combinatorics of faces of associahedron.  We want to emphasize  that  if  we consider the corresponding  associahedral toric variety, using Lemma~\ref{homomor},   we can reformulate this result in terms of the strata of the toric compactification that produces associahedral variety in the similar  way as we do below for the  Losev-Manin, that is permutohedral varieties. Note that if we put $a_n = \frac{\hat{a}_{n}}{n!}$ and $b_n=\frac{\hat{b}_{n}}{n!}$,  the expressions  for  $b_3$ and $ b_4$  from~\cite{BV} right after Theorem 3.3 lead to the following expressions:
\[
\hat{b}_{3} = -\hat{a_3} + 3\hat{a}_{2}^2, \;\; \hat{b}_{4} = -\hat{a}_{4} +10\hat{a}_{2}\hat{a}_{3} - 15\hat{a}_{2}^{3}.
\]
\end{rem}

We want to show, using the results of~\cite{BV},  that the multiplicative inverse of  the exponential power series can be linked to the geometry of moduli space of stable weighted curves $\overline{\mathcal{M}}_{0, \mathcal{A}}$ for $\mathcal{A} = (1, 1, a_3, \ldots , a_n)$, where $\sum _{i\geq 3}a_i \leq 1$.

More precisely,  in~\cite{BV},  it is discussed the relation  between the combinatorics of associahedra and permutohedra, and compositional and multiplicative inversion of formal power series.  In the case of permutohedra Corollary 3.6  from~\cite{BV} describes   the multiplicative inverse power series of a formal power series $f(x) = 1+\sum _{n\geq 1}a_{n}\frac{x^{n}}{n!}$ in terms of combinatorics of  the  permutohedra $P^{n}_{e}$.  

Since the  Losev-Manin spaces $\bar{L}_{0,  n, 2}\cong \overline{\mathcal{M}}_{0, \mathcal{A}}$ for $\mathcal{A} = (1, 1, a_3, \ldots , a_n)$, where $\sum _{i\geq 3}a_i \leq 1$,  are permutohedral toric varieties,  we reformulate this result in terms of these spaces.   Note that the  Losev-Manin spaces $\bar{L}_{0, n, 2}$ provide compactification of the moduli spaces  $\mathcal{M}_{0, n}$.

The following is obvious:

\begin{lem}
The outgrow  $\partial \bar{L}_{0,  n+2, 2} = \bar{L}_{0,  n+2, 2}\setminus (\C ^{\ast})^{n}\cdot c$ for $c\in \mathcal{M}_{0, n}$  is  the union of  $\bar{L}_{n_1, \ldots, n_k} = \bar{L}_{0, n_1+2, 2}\times \cdots \times \bar{L}_{0, n_k+2, 2}$ for all possible $n_1, \ldots ,n_k\geq 2$ such that  $n_1+\ldots +n_k = n$ and $k\geq 2$.
\end{lem}

\begin{proof}
Since $\bar{L}_{0,  n+2, 2}$ is the  permutohedral toric variety over $P_{e}^{n}$,   it follows that $\partial \bar{L}_{0, n, 2}$ is the union of toric varieties over the faces of the permutohedron $P^{n}_{e}$.  It is known that the $(n-k)$ - dimensional faces of $P^{n}_{e}$  are  given  by $P^{n_1}_{e}\times \ldots \times P^{n_k}_{e}$ for all possible partition $\{n_1, \ldots , n_k\}$ of $n$. It follows   $\partial \bar{L}_{0, 2, n+2}$ is the union of $\bar{L}_{0, n_1+2, 2}\times \cdots \times \bar{L}_{0, n_k+2, 2}$.
\end{proof}

We proceed following~\cite{BV}.   Using an identification of  the permutohedral toric varieties  $P_{e}^{n}$ and the Losev-Manin spaces $\bar{L}_{0,  n+1, 2}$,    the generating function 
for  $\bar{L}_{0,  n+1, 2}$ is given  by
\[
F_{\bar{L}_{0, n+2, 2}} (t) = \sum\limits _{\bar{L}_{n_1, \ldots , n_k}\subset \partial \bar{L}_{0, n+2, 2}}\bar{L}_{n_1,\ldots , n_k}t^{n-n_1-\ldots -n_k}
\] 
and the interior by 
\[
\bar{L}^{0}_{0,  n+1, 2} = F_{\bar{L}_{0, n+2, 2}} (-1).
\]

For any $\bar{L}_{n_1, \ldots , n_k}\subset \partial \bar{L}_{0,  n+2, 2}$ we define its combinatorial type  $(m_{n_1}, \ldots , m_{n_k})$  and  the polynomial $m_{n_1, \ldots , n_k} (a) = a_{n_1}^{m_{n_1}}\cdots a_{n_k}^{m_{n_k}}$,  such that the weight $wt(m_{n_1, \ldots , n_k}(a)) = n$, where by definition the weight $wt(a_k)=k$. It easy to verify that these conditions determine uniquely $(m_{n_1}, \ldots , m_{n_k})$. 

Then Corollary 3.6~\cite{BV} implies:

\begin{cor}\label{invperm}
The coefficients of the multiplicative inverse 
\[
g(x) = 1+ \sum\limits_{n\geq 1}b_n\frac{x^{n}}{n!}
\]
of a formal power series $f(x) = 1+\sum _{n\geq 1}a_n\frac{x^{n}}{n!}$ are  by the  following sum
\[
b_n= -\sum\limits _{\bar{L}_{n_1, \ldots , n_k}\subset \partial \bar{L}_{0, n+2, 2}} (-1)^{n-n_1-\ldots -n_k}m_{n_1, \ldots ,n_k}(p).
\]
\end{cor}

\begin{rem}
 For a given exponential series $f(x) = x+\sum _{n\geq 2}a_{n}\frac{x^{n}}{n!}$ and its compositionally inverse $g(x) = x+\sum _{n\geq 2}b_n\frac{x^n}{n!}$ it holds $g(f(x))=x$, which implies that $g^{'}(f(x))\cdot  f^{'}(x) = 1$. Therefore, $g^{'}(f(x))$ is multiplicative inverse to $f^{'}(x) = 1+\sum _{n\geq 1}a_{n-1}\frac{x_{n}}{n!}$,  so by Corollary 3.6 from~\cite{BV}, being equivalent to Corollary~\ref{invperm}, its coefficients are determined by $a_n$ and the outgrows  stratification of the permutohedral varieties.  Together with cited McMullen result from the beginning of this paragraph, this establishes relation  between the outgrows of Deligne-Mumford compactification and toric compactification producing permutohedral variety.
\end{rem}

\begin{ex}\label{ex5}
Let us consider the  $\bar{L}_{0, 5, 2} = \bar{F}_{5}$, that is the toric manifold over $P^{3}_{e}$. The outgrows of the corresponding toric compactification for $(\C ^{\ast})^{2}$-orbit  are given by  six  $(\C ^{\ast})^{1}$-orbits and $6$ points.  The outgrows of the Losev-Manin compactification for  $F_5$ are given by nine   $\mathcal{M}_{0, 4}\times \mathcal{M}_{0, 3}$ and ten $\mathcal{M}_{0,3}^{3}$, compare to~\cite{BT0}.   The outgrows of the Deligne-Mumford compactification consists of ten      $\mathcal{M}_{0, 4}\times \mathcal{M}_{0, 3}$ and fifteen  $\mathcal{M}_{0,3}^{3}$, compare to~\cite{lando}.

Since the Deligne-Mumford compactification is obtained as the blow of  of the Losev-Manin compactification at the point $((1:1), (1:1), (1:1))$ we see, using Proposition~\ref{outgrows} and its proof,  that all outgrows of the Losev-Manin compactification being at the same time outgrows in the toric compactification are the outgrows in the Deligne-Mumford compactification as well.  Thus,   all nine exemplars of  $\mathcal{M}_{0, 4}\times \mathcal{M}_{0, 3}$ remain and the additional one  obviously comes from the blow up procedure. As far  as the outgrow points  are concerned, nine of them  remain   in the Deligne-Mumford compactification, while one disappears due to blowing up.  The new six outgrow points arise as follows.
In the neighborhood  $U$ of  $((1:1), (1:1), (1:1))$ of the form $((1:c_{1}^{'}), (1:c_{2}^{'}), (1:c_{3}^{'}))$ the Deligne-Mumford compactification $\overline{\mathcal{M}}_{0,5}$ is an open submanifold of $\bar{F}_{5}\times \C P^1$  given, in local coordinates $(c_1^{'}, c_{2}^{'})$ by $(1-c_1^{'})x_2=(1-c_2^{'})x_1$, that is 
\[
((1:c_1^{'}), (1:c_2^{'}), (1:c_3^{'}), (1-c_1^{'}: 1-c_2^{'})), \;\; c_2^{'}=c_1^{'}c_3^{'}. 
\]
Its closure is stratified by $\mathcal{M}_{0, 4} =\C P^{1}_{A}$ and the six points $((1:1), (1:0), (1:0),(0:1)),  ((1:1), (0:1),  (0:1), (0:1)), ((1:1), (1:1), (1:1), (0:1)), ((1:0), (1:1), (0:1), (1:0)), ((1:1), (1:1), (1:1), (1:0)), ((0:1), (1:1), (1:0), (1:0))$. 

Note that the outgrow point $((1:1), (1:1), (1:1))$  in the Losev-Manin compactification is not an outgrow point in the Deligne-Mumford compactification, and the given six points in the Deligne-Mumford compactification do not belong to the Losev-Manin compactification.
\end{ex}

\section{Conclusions}

We summarize the results presented in  the previous sections and show that  the  spaces of parameters of the chambers in the model for $G_{n,2}$ provide  o geometric realization of the moduli spaces of weighed stable genus zero curves,  up to birational equivalence.

Recall that birational map from a variety $X$ to a variety $Y$  is a rational map $f : X\dasharrow Y$, such that there is a rational map $Y \dasharrow X$ which is inverse to $f$.  In this case varieties $X$ and $Y$ are said to be birationally equivalent.   A birational morphism is a morphism $f: X \to Y$ which is birational.  
A birational morphism is defined everywhere, but its inverse may not be.

We first note :
\begin{lem}\label{mwbe}
For a fixed $n$ all moduli spaced $\overline{\mathcal{M}}_{0, \mathcal{A}}$, $|\mathcal{A}|=n$, are birationally equivalent.
\end{lem}

\begin{proof}
We provide the proof for the clearness of the exposition. Consider $\overline{\mathcal{M}}_{0, n}= \overline{\mathcal{M}}_{0, \mathcal{A}_{0}}$, where $\mathcal{A}_{0} = (1, \ldots , 1)$ and arbitrary  weight vectors $\mathcal{A} $ and $\mathcal{B}$. Then there exists birational reduction morphisms $\rho _{\mathcal{A}_{0}, \mathcal{A}} : \overline{\mathcal{M}}_{0, \mathcal{A}_{0}} \to \overline{\mathcal{M}}_{0, \mathcal{A}}$ and $\rho _{\mathcal{A}_{0}, \mathcal{B}} : \overline{\mathcal{M}}_{0, \mathcal{A}_{0}} \to \overline{\mathcal{M}}_{0, \mathcal{B}}$. Let $r_{\mathcal{A}}$ and $r_{\mathcal{B}}$ be the corresponding inverse rational maps.  The map $\rho _{\mathcal{A}_{0}, \mathcal{B}}\circ r_{\mathcal {A}} :  \overline{\mathcal{M}}_{0, \mathcal{A}} \to  \overline{\mathcal{M}}_{0, \mathcal{B}}$ is a rational map with the inverse $ \rho _{\mathcal{A}_{0}, \mathcal{B}}\circ r_{\mathcal{B}}$ 
\end{proof}

\begin{rem}
For $n=4$ we have that $\overline{\mathcal{M}}_{0, 4}\cong \C P^1$ and  a  birational morphism $\rho _{\mathcal{A}_{0}, \mathcal{A}}$ being  blown-down is an  identity. Thus,  in this case $\overline{\mathcal{M}}_{0, \mathcal{A}} \cong \C P^1$ for any weight vector $\mathcal{A}$. In particular, it follows that $F_{\omega}\cong \C P^1$  for $C_{\omega}\subset \stackrel{\circ}{\Delta}_{n,2}$, which is showed in~\cite{BT} as well.  Therefore, the first non trivial case may appear  for  $n=5$ and Example~\ref{nohom} below  confirms its non-triviality.
\end{rem}

Note that in general it does not exist any birational   morphism between these  moduli spaces.

Obviously,  such morphisms exist between Losev-Manin spaces:

\begin{lem}\label{LM-birat}
For a fixed $n$ and any $2\leq k_1<k_2\leq n-2$ there exists a  birational morphism $L_{0, n, k_2} \to L_{0, n, k_1}$ between the corresponding Losev-Manin spaces. 
\end{lem}
\begin{proof}
Denote by $\mathcal{A}_{k} = (1,\ldots, 1, a_{k+1}, \ldots, a_n)$ the weight vector for the space $L_{0, n, k}$. It holds  $\mathcal{A}_{2} <\ldots < \mathcal{A}_{n-2}$, which implies that there exists a sequence  of birational morphisms $L_{0, n, n-2} \stackrel{\rho _{\mathcal{A}_{n-2}, \mathcal{A}_{n-3}}}{\lra} L_{0, n, n-3}   \stackrel{\rho _{\mathcal{A}_{n-3}, \mathcal{A}_{n-4}}}{\lra}\ldots  \stackrel{\rho _{\mathcal{A}_{3}, \mathcal{A}_{2}}}{\lra} L_{0, n, 2}$, which proves the statement.
\end{proof}

Theorem\ref{LMChn-1} and Theorem\label{LMChn-2} give:

\begin{prop}\label{chambirat}
For a fixed $n$, all spaces of parameters $F_{\omega}$ of the chambers $C_{\omega}\subset \stackrel{\circ}{\Delta}_{n,2}$ are birationally equivalent.
\end{prop}

\begin{proof}
The proof is analogous to the proof of Lemma~\ref{mwbe}. It follows  from Theorem~\ref{LMChn-2} that  for $\dim C_{\omega _1}, \dim C_{\omega _2}\leq n-2$, we have a  rational map $f : F_{\omega _1}\to L_{0, n,k}$,  which is  inverse to birational morphism $h: L_{0, n, k}\to F_{\omega _1}$  and the  birational morphism $ g : L_{0, n, k}\to F_{\omega _2}$, whose inverse we denote by  $l$.    Thus, it is defined the rational map $g\circ f$, whose inverse is $h\circ l$, so $F_{\omega _1}$ and $F_{\omega _2}$ are birationally equivalent. If $\dim C_{\omega} =n-1$, then  $C_{\omega}$ belongs to a hyperplane $x_{i_1}+ x_{i_2}<1$, since $x_1+\ldots +x_n=2$ and $n\geq 5$. Theorem~\ref{LMChn-1} implies, due to the action of the symmetric group $S_n$, that $F_{\omega}$ is birationally equivalent to $L_{0, n, n-2}$. Then we proceed as in the previous case. 
\end{proof}

We point here as well that in general there is no birational morphism between  spaces of parameters.

\begin{ex}\label{nohom}
We provide an example of two spaces of parameters $F_{\omega _1}$ and $F_{\omega _2}$ which are not homeomorphic. For simplicity, take $n=5$ and a chamber $C_{\omega _1}\subset \stackrel{\circ}{\Delta}_{5,2}$  with a facet on the hyperplane $x_1+x_3=1$ and  such that  $C_{\omega _1}$ belongs to the halfspace $x_1+x_3<1$. Let $C_{\omega _2}$ be a chamber,  which has the common facet with $C_{\omega _1}$ on the hyperplane $x_1+x_3=1$. Then $C_{\omega _2}$ belongs to the half space $x_1+x_3>1$. All admissible polytopes, except that one given by $x_1+x_3\leq 1$,  which contribute to $C_{\omega _1}$ contribute to $C_{\omega _2}$ as well and vice versa, all admissible polytopes, except that one given by $x_1+x_3\geq  1$,  which contribute to $C_{\omega _2}$ contribute to $C_{\omega _1}$.  Thus, $F_{\omega _1} = \cup F_{\sigma}  \cup F_{x_1+x_3\leq 1}$   and $F_{\omega _2} = \cup F_{\sigma}  \cup F_{x_1+x_3\geq 1}$,
Since, it can be easily seen from~\cite{BT1} that  $ F_{x_1+x_3\leq 1}\cong \C P_{1}\setminus \{0, 1, \infty\}$ and  $F_{x_1+x_3\geq 1}$ is a point, it follows that $F_{\omega _1}$ and $F_{\omega _1}$ are not homeomorphic.
\end{ex}

As far as the moduli spaces of stable  weighted  genus zero curves are concerned we prove  that all of them  can be geometrically realized via the  spaces of parameters $F_{\omega}$ generalizing  Corollary~\ref{embedhassett}:


\begin{thm}\label{general}
For any moduli space  $\overline{\mathcal{M}}_{0, \mathcal{A}}$ of stable  $\mathcal{A}$  -  weighted  genus zero curves  there exists a birational morphism  to the  space of parameters $F_{\omega}$   of a chamber $C_{\omega}\subset \stackrel{\circ}{\Delta}_{n,2}$.
\end{thm}

\begin{proof}
Let $D_{\omega}$ be a chamber in $\mathcal{D}_{0, n}$ such that $\mathcal{A}\in D_{\omega}$. If there exists a chamber $C_{\omega}\in \mathcal{C}_{n,2}$ such that $C_{\omega}\subset \overline{D}_{\omega}$, that is $D_{\omega}\in \xi (C_{\omega})$  the statement follows directly from Theorem~\ref{omegaweight} or Theorem~\ref{omegabirat}.

If this is not the case, let $\mathcal{A} = (a_1, \ldots , a_n)$ and $A = \sum _{i=1}^{n}a_i$. Then  $A>2$ and let $\mathcal{B}  =(b_1, \ldots , b_n)$, where $b_i = \frac{2a_i}{A}$. Clearly $b_i<a_i$ and $\sum _{i=1}^{n} b_i=2$. Thus, $\mathcal{B}$ belongs to some chamber $C_{\omega}\subset \stackrel{\circ}{\Delta}_{n,2}$. Let $D_{\omega}\in \mathcal{C}_{\mathcal{H}}$ such that ${D}_{\omega}\in \xi (C_{\omega})$. Then by Theorem~\ref{omegaweight} or Theorem~\ref{omegabirat} there exists a birational morphism $\overline{\mathcal{M}}_{0, \mathcal{E}} \to F_{\omega}$, where $\mathcal{E}\in D_{\omega}$ and $F_{\omega}$ is the space of parameters of $C_{\omega}$. Obviously, $\mathcal{E}$ can be chosen arbitrarily close to $\mathcal{B}$, so we can assume that $e_{i}< a_i$ for $\mathcal{E}= (e_1, \ldots, e_n)$. It follows that there exists a birational reduction morphism $\rho _{\mathcal{A}, \mathcal{E}} : \overline{\mathcal{M}}_{0, \mathcal{A}} \to \overline{\mathcal{M}}_{0, \mathcal{E}}$, which proves the statement.
\end{proof}



\begin{rem}
It follows from Example~\ref{nohom} that there exists  moduli spaces $\overline{\mathcal{M}}_{0, \mathcal{A}}$ with    $\mathcal{A}$ belonging  to the  chambers of maximal dimension in $\mathcal{D}_{0, n}$, such that they  are not homeomorphic.  Precisely,  following notations of Example~\ref{nohom},  since $C_{\omega _1}$ and $C_{\omega _2}$ are of maximal dimension we have that $\overline{\mathcal{M}}_{0, \mathcal{A}_{1}} \cong F_{\omega _1}$ and $\overline{\mathcal{M}}_{0, \mathcal{A}_{2}} \cong F_{\omega _2}$ for $\mathcal{A}_{1}\in D_{\omega _1}$  and $\mathcal{A}_{2}\in D_{\omega _2}$ such that $C_{\omega _1}\subset \overline{D}_{\omega _1}$ and $C_{\omega _2}\subset \overline{D}_{\omega _2}$. It follows that $\overline{\mathcal{M}}_{0, \mathcal{A}_{2}}$ and $\overline{\mathcal{M}}_{0, \mathcal{A}_{2}}$ are not homeomorphic.
\end{rem}

\subsection{Hassett category and Losev-Manin category}

We say that a category $\mathcal{C}_{1}$ can be modeled by a category $\mathcal{C}_{2}$ if there exist birational morphisms from  the objects of  $\mathcal{C}_{1}$ to the objects   of $\mathcal{C}_{2}$, these birational morphisms cover all object from $\mathcal{C}_{2}$  and   they  commute with the morphisms of these categories.   

\subsubsection{Hassett category} Let $\mathcal{H}_{0,n}$ be the category whose objects are  all moduli spaces of weighted stable genus zero curves and the morphisms   are the reduction morphisms $\rho _{\mathcal{A}_{0}, \mathcal{A}}$ .  

We consider the category, which we denote by $\mathcal{U}_{n}$,  whose objects are the universal space of parameters $\mathcal{F}_{n}$ and all spaces of parameters $F_{\omega}$. The morphisms of this category we assume to be the projections $p_{\omega} : \mathcal{F}_{n}\to F_{\omega}$ which are by Proposition~\ref{projred}  defined for any $F_{\omega}$.

Theorems~\ref{general} and  ~\ref{projred} imply:

\begin{thm}
 The  category $\mathcal{H}_{0, n}$ can be modeled by the topological  category $\mathcal{U}_{n}$.
\end{thm}

Let $Y_n= \bigcup \limits_{\dim C_{\omega} =n -1}C_{\omega}$ and $\mathcal{Y}_{n} = Y_n\times \mathcal{F}_{n}\subset U_n$ and $\hat{G}_{n}$ the restriction of the map $G_n$ from $U_n$ to $\mathcal{Y}_{n}$.

Let $\stackrel{\circ}{\mathcal{H}}_{0, n}\subset \mathcal{H}_{0, n}$ be a subcategory whose objects  are   $\overline{\mathcal{M}}_{0, \mathcal{A}_{0}}$ and  those  $\overline{\mathcal{M}}_{0, \mathcal{A}}$ such that $\mathcal{A} \in w_{c}\subset \mathcal{D}_{0,n}$ and $w_{c}\cap \partial \mathcal{D}_{0, n} $ is a chamber of maximal dimension, that is $n-1$. The morphisms in  $\stackrel{\circ}{\mathcal{H}}_{0, n}$ are those induced from  $\mathcal{H}_{0, n}$.
 
Then Theorem~\ref{omegaweight} in particular gives:

\begin{cor}
The  category $\stackrel{\circ}{\mathcal{H}}_{0, n}$ can be isomorphically modeled by the topological category $(\mathcal{Y}_n, \hat{G}_{n})$
\end{cor}

We call  $\stackrel{\circ}{\mathcal{H}}_{0, n}$ the Hassett category.  

Let   $N(T^n)$ be the normalizer of the maximal torus $T^n\subset U(n)$.  Recall that the Grassmannian $G_{n,2}$ is the disjoint union of the strata for the canonical $(\C ^{\ast})^{n}$-action of $G_{n,2}$.  According to~\cite{BT0}, the action of $N(T^n)$  on $G_{n,2}$ permutes the strata  and  induces the action of the symmetric group $S_n$ on the orbit space $G_{n,2}/T^n$. The induced moment map $\hat{\mu} : G_{n,2}/T^n \to \Delta _{n,2}$ is $S_n$ equivariant where $\Delta _{n,2}$ is considered with the standard action of the group $S_n$ permuting the coordinates of points. Thus, the action of $S_n$  on permutes the admissible polytopes  $P_{\sigma}$ in $\Delta_{n, 2}$ as well as their spaces of parameters $F_{\sigma}$ in $G_{n,2}/T^n$. It follows from~\eqref{chdef} and Proposition~\ref{ComegaF} that $S_n$-action on $\Delta _{n,2}$ permutes   chambers $C_{\omega}$ in $\Delta_{n,2}$, while its action on $G_{n,2}/T^n$ permutes the spaces of parameters $F_{\omega}$. 

On the other hand the group $S_n$ acts on the domain $\mathcal{D}_{0,n}$ as well as on its chambers defined by~\eqref{chcoarse}. Thus, $S_n$ acts on the family of moduli spaces $\overline{\mathcal{M}}_{0, \mathcal{A}}$ by permuting the weights.  Being continuous,  this action preserves the family of objects which define  category $\mathcal{H}_{0, n}$. 

Theorem~\ref{omegaweight} and Theorem~\ref{omegabirat} and Theorem~\ref{general} below give:

\begin{prop}
The embedding of  the objects from  $\mathcal{H}_{0,n}$  in $G_{n,2}/T^n$ is equivariant for $S_n$-actions on $\mathcal{H}_{0,n}$ and $G_{n,2}/T^n$. 
\end{prop}

\subsubsection{Losev-Manin category}
Let $\mathcal{L}\mathcal{M}_{0, n}$ be the category whose object are the space $\overline{\mathcal{M}}_{0. n}$ and the   Losev-Manin spaces $\bar{L}_{0, k, n}$, where  $n\geq 5$ and $2\leq k\leq n-2$,  and the morphisms are the reduction morphisms $\overline{\mathcal{M}}_{0,n}\to \bar{L}_{0, n, k}$. We say this category to be the  Losev-Manin category.

 Let $\mathcal{Z}_{n} \subset U_n$ is defined by $\mathcal{Z}_{n}= (\Delta _{n,2}\cap t_3+\ldots +t_n\leq 1) \times \mathcal{F}_{n}$ and $\tilde{G}_{n}$ is the restriction of the map $G_n$ from $U_n$  to $\mathcal{Z}_{n}$.  The objects of the category which we denote by $(\mathcal{Z}_{n}, \tilde{G}_{n})$ are the universal space of parameters $\mathcal{F}_{n}$  and the  spaces of parameters $F_{\omega}$ such that $C_{\omega}\subset (\Delta _{n,2}\cap t_3+\ldots +t_n\leq 1)$,  and the morphisms are the projections $p_{\omega} : \mathcal{F}_{n}\to F_{\omega}$.

Theorem~\ref{LMChn-1}  and Theorem~\ref{LMChn-2} imply:

\begin{thm}
 The Losev-Manin category  $\mathcal{L}\mathcal{M}_{0,n}$ can be modeled by the topological category  $(\mathcal{Z}_{n}, \tilde{G}_n)$.
\end{thm}

Note that categories   $\stackrel{\circ}{\mathcal{H}}_{0, n}$  and $\mathcal{L}\mathcal{M}_{0,n}$ have one common object, that is  $\overline{\mathcal{M}}_{0, n}$.

\subsection{Further investigation}

In general, we can consider   $(2n, k)$ - manifolds~\cite{BT2nk}, that is   smooth, compact manifolds $M^{2n}$ with an effective actions of a torus $T^k$ and  a moment maps 
$\mu :  M^{2n}\to P^{k}$. This moment map gives rise to induced moment map $\hat{\mu} : M^{2n}/T^k \to P^{k}$. According to the existing results in the literature, the following cases differentiate:
\begin{enumerate}
\item $k=n$ and  $\hat{\mu}^{-1}(x)$ is a point for $x\in P^{k}$. Examples of such manifolds are toric and quasitoric manifolds, see~\cite{BP}.
\item  $k=n-1$ and $\hat{\mu}^{-1}(x)\cong \hat{\mu}^{-1}(y)$ for all $x, y \in \stackrel{\circ}{P^{k}}$. Among examples for this case are, see~\cite{BT},~\cite{BT2nk} ~\cite{AY},~\cite{AY1}, \cite{YK}: the canonical $T^4$ - action on $G_{4,2}$, the canonical   $T^3$ - action on the complete flag $F_{3} = U(3)/T^3$,  the canonical $T^3$-action on the quaternionic projective plane, $T^2$-action on the sphere $S^6$,  Hamiltonian $T^{n-1}$ - action on a symplectic manifold $M^{2n}$,
\item $k\leq n-2$ and there exist $x, y\in \stackrel{\circ}{P^{k}}$ such that $\hat{\mu}^{-1}(x)\not\cong \hat{\mu}^{-1}(y)$. Examples for this case are provided by the canonical $T^n$-action on $G_{n,2}$ for $n\geq 5$, see Example~\ref{nohom} and~\cite{BT0}.
\end{enumerate}

We already emphasized  by  Example~\ref{nohom}  that the spaces $F_{\omega}$, $C_{\omega} \subset \stackrel{\circ}{\Delta}_{n,2}$  are not all  homeomorphic for $n\geq 5$. 
It naturally arises the following problem, which can be considered as geometric complexity of a torus action.

\begin{pr}  
Classify the spaces $\hat{\mu}^{-1}(x)$, $x\in \stackrel{\circ}{P^{k}}$ up to homeomorphism or birational equivalence,
\end{pr}

According to our results, for  the Grassmannians $G_{n,2}$ this problem is equivalent to:

\begin{pr}
 Formulate classification criterion for $F_{\omega}$ in terms of Hassett and Losev-Manin theory. 
\end{pr}

\begin{rem} We want to point out  that for an effective  algebraic torus $(\C ^{\ast})^{k}$-action on a smooth, compact manifold $M^{2m}$, $k\leq m$, it is defined the notion of  complexity   to be the number $d=m-k$, see for example~\cite{tim}.  The complexity of  the canonical $(\C ^{\ast})^{n}$ -action  on  $M^{2m}=G_{n,2}$, $m=2(n-2)$  is $d=n-1$.  Note that the spaces of parameter over the chambers are $(\C ^{\ast})^{n}$ - quotients, that is 
\begin{equation}\label{omebaconcl}
F_{\omega} = \cup _{\sigma \in \omega}W_{\sigma}/(\C ^{\ast})^{n})  \cong \hat{\mu}^{-1}(x), \; x\in C_{\omega},
\end{equation}
and $\dim  \cup _{\sigma \in \omega}W_{\sigma} = \dim G_{n,2}$.
In this way the research of this paper is related  to developing of toric topology of positive complexity.
\end{rem}

\bibliographystyle{amsplain}

 Victor M.~Buchstaber\\
Steklov Mathematical Institute, Russian Academy of Sciences\\ 
Gubkina Street 8, 119991 Moscow, Russia\\
E-mail: buchstab@mi.ras.ru
\\ \\ 

Svjetlana Terzi\'c \\
Faculty of Science and Mathematics, University of Montenegro\\
Dzordza Vasingtona bb, 81000 Podgorica, Montenegro\\
E-mail: sterzic@ucg.ac.me 

\end{document}